\documentclass[reqno,11pt]{amsart}

\usepackage{amscd}
\usepackage{amsmath}
\usepackage{amssymb}
\usepackage[american]{babel}
\usepackage{bbm}
\usepackage{bookmark}
\usepackage{cmap} 
\usepackage{dsfont}
\usepackage{enumerate}
\usepackage{epigraph}
\usepackage[mathscr]{euscript}
\usepackage[myheadings]{fullpage}
\usepackage{graphicx}
\usepackage[geometry]{ifsym}
\usepackage{mathabx}
\usepackage{accents}
\usepackage{mathrsfs}
\usepackage{mathtools}
\usepackage{refcount}
\usepackage{setspace}
\usepackage{stmaryrd}
\usepackage{thinsp}
\usepackage{verbatim}
\usepackage[all,2cell]{xy} \UseTwocells
\usepackage{xr}
\usepackage[multiple]{footmisc}

\usepackage{color}
\usepackage{soul}

\usepackage{hyperref}

\numberwithin{equation}{subsection}

\newtheorem{thm}{Theorem}[subsection]
\newtheorem*{thm*}{Theorem}

\newtheorem{cor}[thm]{Corollary}
\newtheorem*{cor*}{Corollary}
\newtheorem{lem}[thm]{Lemma}

\newtheorem{prop}[thm]{Proposition}
\newtheorem{prop-const}[thm]{Proposition-Construction}

\newtheorem*{conjecture*}{Conjecture}

\newtheorem*{princ*}{Principle}

\theoremstyle{remark}
\newtheorem{rem}[thm]{Remark}
\newtheorem{example}[thm]{Example}

\newtheorem{defin}[thm]{Definition}

\newtheorem{notation}[thm]{Notation}

\newtheorem{variant}[thm]{Variant}

\newcounter{steps}[thm]
\newtheorem{step}[steps]{Step}

\newcommand{\xar}[1]{\xrightarrow{#1}}
\newcommand{\rar}[1]{\xar{#1}}
\newcommand{\isom}{\rar{\simeq}}

\newcommand{\rightrightrightarrows}{%
        \mathrel{\vcenter{\mathsurround0pt
                \ialign{##\crcr
                        \noalign{\nointerlineskip}$\rightarrow$\crcr
                        \noalign{\nointerlineskip}$\rightarrow$\crcr
                        \noalign{\nointerlineskip}$\rightarrow$\crcr
                }%
        }}%
}

\newcommand{\into}{\hookrightarrow}
\newcommand{\onto}{\twoheadrightarrow}

\newcommand{\bDelta}{\mathbf{\Delta}}

\newcommand{\bB}{{\mathbb B}}

\newcommand{\bD}{{\mathbb D}}

\newcommand{\bF}{{\mathbb F}}
\newcommand{\bG}{{\mathbb G}}

\newcommand{\bL}{{\mathbb L}}

\newcommand{\bQ}{{\mathbb Q}}

\newcommand{\bS}{{\mathbb S}}
\newcommand{\bT}{{\mathbb T}}

\newcommand{\bZ}{{\mathbb Z}}

\newcommand{\sC}{{\EuScript C}}
\newcommand{\sD}{{\EuScript D}}
\newcommand{\sE}{{\EuScript E}}
\newcommand{\sF}{{\EuScript F}}
\newcommand{\sG}{{\EuScript G}}

\newcommand{\sI}{{\EuScript I}}
\newcommand{\sJ}{{\EuScript J}}

\newcommand{\sS}{{\EuScript S}}

\newcommand{\sX}{{\EuScript X}}

\newcommand{\on}{\operatorname}

\renewcommand{\ul}{\underline}

\newcommand{\mathendash}{\text{\textendash}}

\newcommand{\Ker}{\on{Ker}}
\newcommand{\Coker}{\on{Coker}}
\newcommand{\End}{\on{End}}
\newcommand{\Hom}{\on{Hom}}

\newcommand{\Aut}{\on{Aut}}

\newcommand{\id}{\on{id}}

\newcommand{\tr}{\on{tr}}

\newcommand{\Rep}{\mathsf{Rep}}
\newcommand{\fil}{\on{fil}}
\newcommand{\gr}{\on{gr}}

\renewcommand{\dot}{\bullet}

\newcommand{\vph}{\varphi}
\newcommand{\vareps}{\varepsilon}

\newcommand{\Sym}{\on{Sym}}

\renewcommand{\mod}{\mathendash\mathsf{mod}}

\newcommand{\bimod}{\mathendash\mathsf{bimod}}
\newcommand{\Proj}{\mathsf{Proj}}

\newcommand{\colim}{\on{colim}}

\newcommand{\StCat}{\mathsf{StCat}}
\newcommand{\Sp}{\mathsf{Sp}}

\newcommand{\fSet}{\mathsf{fSet}}
\newcommand{\Gpd}{\mathsf{Gpd}}

\renewcommand{\lim}{\on{lim}}
\newcommand{\Ind}{\mathsf{Ind}} 

\newcommand{\TwoHom}{\mathsf{Hom}}
\newcommand{\TwoEnd}{\mathsf{End}}

\newcommand{\SqZero}{\on{SqZero}}
\newcommand{\THH}{\on{THH}}
\newcommand{\TC}{\on{TC}}
\newcommand{\TP}{\on{TP}}

\newcommand{\heart}{\heartsuit}

\newcommand{\Oblv}{\on{Oblv}}

\newcommand{\Alg}{\mathsf{Alg}}

\renewcommand{\subset}{\subseteq}

\makeatletter
\newcommand{\biggg}{\bBigg@{4}}
\newcommand{\Biggg}{\bBigg@{5}}
\makeatother

\date{\today}

\begin{document}

\frenchspacing

\setlength{\epigraphwidth}{0.4\textwidth}
\renewcommand{\epigraphsize}{\footnotesize}

\title{On the Dundas-Goodwillie-McCarthy theorem}

\author{Sam Raskin}

\address{The University of Texas at Austin, 
Department of Mathematics, 
RLM 8.100, 2515 Speedway Stop C1200, 
Austin, TX 78712}

\email{sraskin@math.utexas.edu}

\begin{abstract}

We give a modern presentation of the Dundas-Goodwillie-McCarthy theorem 
identifying relative $K$-theory and topological cyclic homology for nilpotent
ring extensions. 

\end{abstract}

\maketitle

\setcounter{tocdepth}{1}
\tableofcontents

\section{Introduction}

\subsection{} This paper proves the 
following theorem of 
Dundas-Goodwillie-McCarthy, relating $K$-theory and
topological cyclic homology $\TC$:

\begin{thm}[\cite{dgm} Theorem 7.2.2.1]\label{t:dgm}

For $B \to A$ a morphism of connective $\sE_1$-ring spectra
such that $\pi_0(B) \to \pi_0(A)$ is surjective
with kernel a nilpotent ideal,
the \emph{cyclotomic trace} map $K \to \TC$ 
induces an equivalence of spectra:

\[
\Ker(K(B) \to K(A)) \to \Ker(\TC(B) \to \TC(A)).
\]

\end{thm}

Here we use $\sE_1$ to refer to the appropriate notion of associative
algebra in the homotopical setting; $A_{\infty}$ and \emph{highly structured}
are common synonyms.  
We refer to \cite{nikolaus-scholze} for an introduction to topological
cyclic homology. 

As we discuss at greater length below in \S \ref{ss:dgm-comparison}, 
one of the main purposes of this note is to simplify the original arguments
of Dundas-Goodwillie-McCarthy through systematic use of Lurie's 
$\infty$-categorical methods (c.f. \cite{htt} and \cite{higheralgebra})
and using the recent ideas of Nikolaus-Scholze.

\begin{rem}

Theorem \cite{dgm} has a long history. 
To the extent that
cyclic homology is an avatar for de Rham (or crystalline)
cohomology, its origin is in \cite{bloch}.

Goodwillie proved a rationalized
version of Theorem \ref{t:dgm} in \cite{goodwillie}. The connection between
$K$-theory and topological Hochschild homology was first
proved in \cite{dundas-mccarthy} and \cite{dundas-mccarthy-erratum};
see also \cite{ssw}.
The cyclotomic trace\footnote{We remind that this is a natural map
$K(A) \to \TC(A)$ for any $\sE_1$-algebra $A$. More generally, this
map is defined for any essentially small stable category $\sC$ (and
in particular is Morita invariant).}
was constructed in \cite{bhm}
(though see \cite{bgt} for a simpler construction).
A $p$-adic version of Theorem \ref{t:dgm} was proved in \cite{mccarthy}.
Lindenstrauss-McCarthy proved the theorem for 
split square-zero extensions (and usual rings)
in \cite{lindenstrauss-mccarthy}. 

For more recent developments, see \cite{cmm} and its application in \cite{bms2} \S 7.

\end{rem}

\begin{rem}

It is not immediately clear how to deduce Goodwillie's 
rationalized version \cite{goodwillie} of Theorem \ref{t:dgm}
from Theorem \ref{t:dgm} (because $\TC$ does not generally commute
with tensor products).
Our Theorem \ref{t:rat'l} provides the relevant comparison. 

Similarly, Beilinson \cite{beilinson} proved a rationalized
version of Theorem \ref{t:dgm} in a $p$-adic setting: we refer
to \emph{loc. cit}. for the formulation.
Again, Beilinson's theorem does not obviously
follow from Dundas-Goodwillie-McCarthy. Forthcoming  
work of Ben Antieau, Akhil Mathew, and Thomas Nikolaus
explains the deduction of Beilinson's theorem from 
Theorem \ref{t:dgm}.
 
\end{rem}

\subsection{Outline of the argument}\label{ss:pf-outline}

The proof of Theorem \ref{t:dgm} 
is quite striking: one shows that the cyclotomic trace
induces an isomorphism on \emph{Goodwillie derivatives}, and then
formally deduces Theorem \ref{t:dgm} 
from \emph{structural properties} of the functors $K$-theory and $\TC$.

Let us explain what we mean here in more detail. 
Fix a connective $\sE_1$-algebra $A$.
Then for any $A$-bimodule $M$, we can form the split square-zero extension 
$A \oplus M$ and take its $K$-theory. 
The induced functor $A\bimod \to \Sp$ is not additive for a stupid
reason: it does not map $0$ to $0$. But even the less naive functor
$M \mapsto \Ker(K(A\oplus M) \to K(A))$ is not additive.
Goodwillie's derivative construction stabilizes this latter
construction to produce a new functor that does commute with
direct sums, and in fact, all colimits. This Goodwillie derivative
of $K$-theory is often called \emph{stable $K$-theory}, and was
first studied by Waldhausen in \cite{waldhausen-stable}.

The theorem \cite{dundas-mccarthy} of Dundas-McCarthy in fact
identifies stable $K$-theory up to a cohomological shift
with the functor $\THH(A,-):A\bimod \to \Sp$.
Here $\THH$ denotes \emph{topological Hochschild homology}
(which is possible meaning of Hochschild homology
in the setting of spectra).

The same constructions may be applied to $\TC$. One again
shows that ``stable $\TC$" coincides with $\THH$ up to shift,
and that this identification is compatible with the cyclotomic trace;
in particular, the cyclotomic trace induces an isomorphism on Goodwillie 
derivatives.
(The former result is due to Hesselholt \cite{hesselholt-stable-tc},
while the compatibility with the cyclotomic trace map
seems to be due to \cite{lindenstrauss-mccarthy} \S 11.) 

The \emph{structural properties} we referred to are results about 
$K$-theory and $\TC$ commuting with certain colimits and 
limits and having some cohomological
boundedness properties. 
The key point is that these structural properties are features of the 
two functors considered separately. 

To reiterate: we are proving a theorem
about the \emph{fiber} of the cyclotomic trace map
$K \to \TC$, which looks like a quite subtle object,
it suffices to know its Goodwillie derivative and
certain structural properties of $K$-theory and $\TC$ independently.

\subsection{}

The method for reducing Theorem \ref{t:dgm} to its stabilized form
is called \emph{Goodwillie calculus}. 
The interested reader may refer to \cite{goodwillie-icm} for
an overview of this subject, and \cite{higheralgebra} \S 6
for a thorough treatment of the subject; however, the 
present note is self-contained in terms of what we
use from Goodwillie calculus.

\subsection{What do we need to know about $K$-theory and $\TC$?}

There is a remarkable asymmetry in how $K$-theory and $\TC$ are treated.

In the proof of Theorem \ref{t:dgm}, we essentially never
calculate anything about $K$-theory.
The argument that its derivative is $\THH$ up to shift is quite soft
(see the proof of Theorem \ref{t:k-deriv}). Establishing sufficient
structural properties of $K$-theory to obtain Theorem \ref{t:dgm}
for split square-zero extensions is also not terribly difficult (and is contained
in \S \ref{s:k}). For the general form of Theorem \ref{t:dgm}, 
some more subtle methods are needed to establish the
necessary structural properties (see Theorem \ref{t:volodin}).

In contrast, our knowledge of $\TC$ is based almost entirely around an explicit
formula for $\TC$ in the case of split square-zero extensions: 
see Theorem \ref{t:tc-calc}.

This disparity indicates the well-established strength of
Theorem \ref{t:dgm} for $K$-theory calculations.

\subsection{Comparison with \cite{dgm}}\label{ss:dgm-comparison}

The exposition here closely follows \cite{dgm} at some points,
but departs from it in several notable respects.

First, our treatment of topological cyclic homology 
follows the recent work of \cite{nikolaus-scholze}, which
clarifies the construction of $\TC$.
These notes assume
some substantial familiarity with their approach to $\TC$, and we take
Nikolaus and Scholze's constructions as definitions. 
This allows us to circumvent equivariant homotopy theory here 
(although we sometimes draw notation from that subject).

We also give a somewhat streamlined approach to 
Goodwillie's calculus of functors in \S \ref{s:goodwillie}.
In particular, we avoid (or at least mask) the analysis of cubes 
that appears in \cite{dgm} \S 7.2.1.2. 

In addition, we use Lurie's higher methods
(\cite{htt} and \cite{higheralgebra}), which
provide simpler homotopical foundations than those used in 
\cite{dgm}. In particular, Lurie's approach
makes it routine to treat $\sE_1$-algebras more directly than in \cite{dgm},
where the authors avoid certain homotopy coherence questions
by reducing to simplicial algebras with some
cost to the conceptual clarity. We treat
general connective $\sE_1$-algebras on equal footing with
classical (alias: discrete) rings, and have avoided reducing
problems to (non-topological) Hochschild homology for rings.

In particular, we use higher categorical methods to 
directly apply Goodwillie's calculus to functors
from $A$-bimodules to spectra.

As a final difference with \cite{dgm}, which may be
more neutral than an improvement, we
have also chosen to work in an abstract categorical
setting where possible. That is, where possible we work with 
(suitable) categories $\sC$ that would be $A\mod$ in cases of interest.
This is not because it is so important for applications 
to work in this generality, but because the author personally finds this
setting to be clarifying, and for the sake of diversifying the literature
somewhat. 
Other expositions tend to work
directly with algebras and their modules, and the reader
may readily find arguments in that more familiar language 
in the literature.

\subsection{Structure of these notes}

In \S \ref{s:goodwillie}-\ref{s:tc}, we treat the split square-zero case of
Theorem \ref{t:dgm}. In \S \ref{s:goodwillie}, we explain how to reduce
the theorem in this case to the comparison of derivatives and structural
properties of these functors. In \S \ref{s:k} and \ref{s:tc}, we prove the
corresponding facts about $K$-theory and $\TC$. 

Finally, in \S \ref{s:gen'l}, we axiomatize the additional structural
facts about $K$-theory and $\TC$ needed for the general
form of Theorem \ref{t:dgm}, and then we establish these features. 

\subsection{Categorical notions}

We systematically use higher category theory and higher algebra in our
treatment, following \cite{htt} and \cite{higheralgebra}. 
We find it convenient to avoid ``higher" terminology everywhere,
so all terminology should be understood in its homotopical form: 
\emph{category} means \emph{$(\infty,1)$-category}, \emph{colimit} means
\emph{homotopy colimit}, and so on. In this spirit, we 
refer to the stable $\infty$-categories of \cite{higheralgebra} simply
as \emph{stable categories}.

We let $\Gpd$ denote the category of (higher) groupoids,
i.e., ``spaces" in more standard homotopy-theoretic language.

We casually use an equals sign between two objects in a category
to indicate the presence of a (hopefully) clear isomorphism. 

\subsection{}

For a category $\sC$ and $\sF,\sG \in \sC$, 
we let $\Hom_{\sC}(\sF,\sG) \in \Gpd$ denote
the groupoid of maps from $\sF$ to $\sG$.

For categories $\sC$ and $\sD$, we let $\TwoHom(\sC,\sD)$ denote
the category of functors from $\sC$ to $\sD$. More generally,
use $\TwoHom$ to indicate the category of 
functors in a $2$-category (meaning $(\infty,2)$-category).

\subsection{}

For $\sC$ stable and $\sF \in \sC$, we use the notations
$\sF[1]$ and $\Sigma \sF$ (resp. $\sF[-1]$ and $\Omega\sF$) interchangeably,
capriciously choosing between the two.  

Following the above conventions, for
$f: \sF \to \sG \in \sC$, we let $\Ker(f)$ denote
the \emph{homotopy} kernel (or \emph{fiber}) of $f$, and let
$\Coker(f)$ denote the \emph{homotopy} cokernel (or \emph{cofiber}, or \emph{cone}) of $f$.
Recall that $\Ker(f)[1] = \Coker(f)$.

For a $t$-structure on $\sC$, we use cohomological notation: 
$\sC^{\leq 0} \subset \sC$ are the connective objects, and $\sC^{\geq 0}$ are
the coconnective objects. We let $\tau^{\leq n}$ and $\tau^{\geq n}$ denote
the corresponding truncation functors. 

\subsection{}

Let $\StCat_{cont}$ denote the category of cocomplete\footnote{We omit set-theoretic
considerations. So \emph{cocomplete} should be taken to mean \emph{presentable}.
Similarly, all functors between accessible categories will themselves be assumed
accessible (i.e., wherever this hypothesis is reasonable we assume it).
When we refer to commutation with all colimits, we mean small colimits.}
stable categories under continuous\footnote{We say a 
functor is \emph{continuous} if it commutes with filtered
colimits.}
exact functors (i.e.,
functors commuting with 
all colimits).

We recall from \cite{higheralgebra} \S 4.8 that $\StCat_{cont}$
has a symmetric monoidal structure $\otimes$. The basic property is that
a functor $\sC \otimes \sD \to \sE \in \StCat_{cont}$ is equivalent to a functor
$\sC \times \sD \to \sE$ that commutes with colimits in each variable separately.
We denote the canonical (non-exact!) functor $\sC \times \sD \to \sC \otimes \sD$
by $(\sF,\sG) \mapsto \sF \boxtimes \sG$.

We use this tensor product quite substantially. We refer the reader to
\cite{grbook} Chapter I.1 for the relevant background material.

\subsection{} 

We let $\Sp \in \StCat_{cont}$ denote the category of spectra,
which we recall is the unit for the symmetric monoidal structure on $\StCat_{cont}$.
We let $\otimes: \Sp \times \Sp \to \Sp$ denote the tensor (alias: smash) product,
and we let $\bS \in \Sp$ denote the sphere spectrum.
We have the standard adjunction $\Sigma^{\infty}:\Gpd \rightleftarrows \Sp:\Omega^{\infty}$.

For $\sC$ a stable category and $\sF,\sG \in \sC$, we let
$\ul{\Hom}_{\sC}(\sF,\sG)$ denote the spectrum of maps from 
$\sF$ to $\sG$. We use $\Hom_{\sC}(\sF,\sG)$ to mean the
groupoid of maps in the category $\sC$, forgetting it was stable.
In other words, $\Hom_{\sC}(\sF,\sG) = \Omega^{\infty}\ul{\Hom}_{\sC}(\sF,\sG)$.

\subsection{}

We say $\sC \in \StCat_{cont}$ is \emph{dualizable} if it is dualizable with
respect to the above symmetric monoidal structure. We let the dual
category be denoted $\sC^{\vee}$. 

In this case, for any $\sD \in \StCat_{cont}$ we have
$\sC^{\vee} \otimes \sD \isom \TwoHom_{\StCat_{cont}}(\sC,\sD)$.
In particular, $\sC^{\vee} = \TwoHom_{\StCat_{cont}}(\sC,\Sp)$. 
Moreover, we see that there is an evaluation map 
$\TwoEnd_{\StCat_{cont}}(\sC) = \sC\otimes \sC^{\vee} 
\to \Sp \in \StCat_{cont}$,
which we denote by $\tr_{\sC}$ and refer to as the \emph{trace}.

For $\sF \in \sC$ a \emph{compact} object, the functor (by fiat) 
$\ul{\Hom}_{\sC}(\sF,-):\sC \to \Sp$ commutes with colimits, so defines
an object of $\sC^{\vee}$. We denote this object by $\bD \sF$.

\begin{example}

If $\sC$ is compactly generated, then $\sC$ is dualizable. Explicitly,
if $\sC^c \subset \sC$ is the (essentially small) subcategory of
compact objects in $\sC$, then $\Ind(\sC^{c,op}) = \sC^{\vee}$ 
(for $\Ind$ denoting the ind-category), where the underlying functor
$\sC^{c,op} \to \sC^{\vee}$ is the map $\sF \mapsto \bD \sF$ above.
(See \cite{grbook} Chapter I.1 for more details.)

\end{example}

\begin{example}

For $\sC = A\mod$ in the above, note that $A\mod$ is compactly generated
by perfect $A$-modules, so $A\mod$ is dualizable. Moreover,
duality gives a contravariant equivalence between left and right $A$-modules,
so $A\mod^{\vee} = A^{op}\mod$ (modules over $A$ with the opposite multiplication).

One then obtains a standard (Morita-style) identification
$\TwoEnd_{\StCat_{cont}}(A\mod) = A\mod \otimes A^{op}\mod = A\bimod$.
The trace map constructed above then corresponds to (topological) Hochschild
homology.

\end{example}
 
\subsection{}

We frequently reference \emph{sifted} colimits,
which may not be familiar to all readers. We review
the theory briefly here, referring to \cite{htt} \S 5.5.8
for proofs.

A category $I$ is \emph{sifted} if it is non-empty and $I \to I \times I$ is
cofinal. A functor commutes with sifted colimits if and only if
it is continuous and commutes with geometric realizations (i.e., colimits
of simplicial diagrams). 
A functor $F$ commutes with all colimits 
if and only if it commutes with sifted colimits and finite coproducts.

For $\sC \in \StCat_{cont}$, the functors
$\sC \xar{\sF \mapsto \sF^{\boxtimes n}} \sC^{\otimes n}$
are typical examples of functors
that commute with sifted colimits but not (outside obvious
exceptional cases) general colimits.

\subsection{Acknowledgements}

These notes were written to supplement a talk I gave in the
Arbeitsgemeinschaft on topological cyclic homology at MFO in April 2018.
I am grateful to the organizers for the opportunity to learn this subject and to
speak about it.

I also would like to thank Sasha Beilinson, Lars Hesselholt,
Akhil Mathew, and Thomas Nikolaus
for helpful discussions on these subjects, for
encouragement to write these notes, and for their comments and corrections.

\section{Goodwillie calculus}\label{s:goodwillie}

\subsection{}

This goal for this section is to setup the proof of 
Theorem \ref{t:dgm} in the split square-zero case.
We end the section by formulating Theorems \ref{t:k} and \ref{t:tc},
which are about $K$-theory and $\TC$ respectively; although the
proofs of these theorems are deferred to later sections,
we deduce the split square-zero case of Theorem \ref{t:dgm}
from them in \S \ref{ss:split}.

However, most of this section is devoted to some 
formal vanishing results, allowing us to carry out the
strategy indicated in \S \ref{ss:pf-outline}.
Our main results here are Theorem \ref{t:conn-algs}
and Corollary \ref{c:full-vanishing}; the former is a toy model
for the latter.
The arguments are essentially the same 
in the two cases, but Theorem \ref{t:conn-algs} is technically simpler
to formulate, and its proof essentially 
leads to the formulation of the more technical Corollary \ref{c:full-vanishing}. 

\begin{rem}

We have sought a minimalist approach to this material
and have omitted many lovely aspects of Goodwillie's theory
here: most notably, higher derivatives, the Goodwillie tower, and the analogy
with calculus. This makes our treatment somewhat non-standard, and
we refer the reader to the extensive literature in this subject
(for example, \cite{goodwillie-icm} and \cite{higheralgebra} \S 6)
for a more thorough approach.

\end{rem}

\begin{rem}

The next comments are intended for the reader who wishes to tighten the connection
between this section and Goodwillie's theory.
Theorems \ref{t:diff'l} and \ref{t:full-vanishing}, which provide hypotheses
for a functor to vanish, can be deduced from Goodwillie's theory by standard methods.
Indeed, one would prove that the Goodwillie tower converges for functors
satisfying these hypotheses (c.f. Remark \ref{r:analytic}), 
and that the hypotheses imply that
all higher Goodwillie derivatives vanish.

The methods we use below in proving these theorems 
are not so different from standard ones in Goodwillie's theory.
However, the Goodwillie tower and higher derivatives take some work to
set up, so we prefer to circumvent these constructions. 

\end{rem}

\subsection{A convention}\label{ss:convention}

The reader may safely skip the present discussion.

Throughout this section, we generally consider functors
between stable categories (or their subcategories) 
commuting with sifted colimits.
This choice is not at all because this assumption is essential.
These assumptions can be significantly relaxed, and we refer
to \cite{higheralgebra} \S 6 for an approach with minimal hypotheses.
Note, for example, that the discussion of \S \ref{ss:goodwillie-deriv}
carries through as is if we work with functors commuting
with $\bZ^{\geq 0}$-indexed colimits.

We include this hypothesis because the commutation with geometric
realizations is essential in our approach;
the additional commutation with all filtered colimits is fairly minor;
and the commutation with sifted colimits motivates the
discussion of \S \ref{ss:bilinear}.

\subsection{The Goodwillie derivative}\label{ss:goodwillie-deriv}

Suppose $\psi:\sC \to \sD$ is a functor between cocomplete stable categories
commuting with sifted colimits.
The \emph{Goodwillie derivative}
$\partial \psi:\sC \to \sD$ of $\psi$ is initial
among continuous exact functors receiving a natural
transformation from $\psi$.

\begin{rem}

By \cite{higheralgebra} Proposition 1.4.2.13,
$\partial \psi$ is calculated as follows.

First, observe that for any $\sF \in \sC$, 
the morphisms $0 \to \sF \to 0$ give a functorial direct sum
decomposition $\psi(\sF) = \psi_{red}(\sF) \oplus \psi(0)$.
Then $\psi_{red}$ is \emph{reduced}, i.e., it takes
$0 \in \sC$ to $0 \in \sD$.

Then there is a canonical natural transformation 
$\Sigma \circ \psi_{red} \to \psi_{red} \circ \Sigma$, or 
equivalently,
$\psi_{red} \to \Omega \psi_{red} \Sigma$. Finally,
we have:

\[
\partial \psi = {\colim} \, 
\Big( \psi_{red} \to \Omega \psi_{red} \Sigma \to 
\Omega^2 \psi_{red} \Sigma^2 \to \ldots \Big) 
\] 

\end{rem}

We will actually use this construction in slightly more generality.

\begin{variant}

Suppose $\sC$ is equipped with a $t$-structure compatible with
filtered colimits and we are given $\psi:\sC^{\leq 0} \to \sD$ commuting
with sifted colimits. Then there is again a functor 
$\partial \psi: \sC \to \sD \in \StCat_{cont}$ 
initial among functors commuting with colimits and receiving
a natural transformation $\psi \to (\partial \psi)|_{\sC^{\leq 0}}$.

To construct $\partial \psi$ in this setup, note that
$\psi_{red}$ makes sense as before, and then one has:

\[
\partial \psi(\sF) = \underset{m}{\colim} \, \underset{n\geq m}{\colim} \,
\Omega^n\psi_{red}(\Sigma^n\tau^{\leq m}\sF)
\]

\end{variant}

If the $t$-structure on $\sC$ is right complete and
$\psi:\sC \to \sD$ commutes with sifted colimits,
then $\partial \psi$ in the previous sense
coincides with $\partial(\psi|_{\sC^{\leq 0}})$.

\subsection{Notation}

We let $\Alg$ denote the category of $\sE_1$-algebras.
We let $\Alg_{conn} \subset \Alg$ denote the subcategory of 
connective $\sE_1$-algebras. 

Recall that for
$A \in \Alg$ and $M$ an $A$-bimodule,
we can form the \emph{split square-zero extension} $\SqZero(A,M) \in \Alg$,
whose underlying spectrum is $A\oplus M$. 

\subsection{}

The following result is a first approximation to the main result of this
section. 

\begin{thm}\label{t:conn-algs}

Let $\Psi: \Alg_{conn} \to \Sp$ be a functor.

Suppose that for every $A \in \Alg_{conn}$, the
functor:

\[
\begin{gathered}
\Psi_A:A\bimod^{\leq 0} \to \Sp \\
M \mapsto \Psi(\SqZero(A,M))
\end{gathered}
\]

\noindent
commutes with sifted colimits and
has vanishing Goodwillie derivative. 
Suppose moreover that its underlying reduced functor 
$\Psi_{A,red}$ maps $A\bimod^{\leq 0}$
to $\Sp^{\leq 0}$.

Then for every $A \in \Alg_{conn}$ and $M \in A\bimod^{\leq -1}$,
the map $\Psi(\SqZero(A,M)) \to \Psi(A)$ is an isomorphism.

\end{thm}

\begin{rem}

In this result, $\Sp$ may be replaced by any $\sD \in \StCat_{cont}$
with a left separated $t$-structure.

\end{rem}

\subsection{A categorical variant}

We will deduce Theorem \ref{t:conn-algs}
from the following result.

\begin{thm}\label{t:diff'l}

Let $\sC$ and $\sD$ be cocomplete stable categories
equipped with $t$-structures compatible with filtered colimits, 
and suppose the $t$-structure
on $\sD$ is left separated.
Let $\psi:\sC^{\leq 0} \to \sD^{\leq 0}$ be a 
reduced functor commuting with sifted colimits.
Suppose that\footnote{For clarity:
the notation indicates the Goodwillie derivative
of the functor $\sG \mapsto \psi(\sF \oplus \sG)$.}
$\partial(\psi(\sF \oplus -)) = 0$
for every $\sF \in \sC^{\leq -1}$.

Then $\psi(\sF) = 0$ for every $\sF \in \sC^{\leq -1}$.

\end{thm}

\begin{proof}[Proof that Theorem \ref{t:diff'l} implies Theorem \ref{t:conn-algs}]

Fix $A \in \Alg_{conn}$ and define $\psi:A\bimod^{\leq 0} \to \Sp$ 
as $\Psi_{A,red}$. 
We claim that the hypotheses of Theorem \ref{t:diff'l} are satisfied.

There are explicit assumptions in Theorem \ref{t:conn-algs} that
$\psi$ commutes with sifted colimits and maps into $\Sp^{\leq 0}$.

We need to show that for $M \in A\bimod^{\leq -1}$,
the Goodwillie derivative of $\psi(M \oplus -)$ is zero;
in fact, we will show this for $M \in A\bimod^{\leq 0}$.
For $M = 0$, this is an assumption.
In general, we have:

\[
\psi(M \oplus -) \oplus \Psi(A) = \Psi(\SqZero(A,M\oplus -)) =
\Psi(\SqZero(\SqZero(A,M),-)).
\]

\noindent The latter functor has vanishing derivative by the
hypothesis that the Goodwillie derivative of
$\Psi_{\SqZero(A,M)}$ vanishes.

\end{proof}

\begin{rem}

Akhil Mathew communicated the following example to us,
showing that Theorem \ref{t:diff'l} is sharp. Let 
$\sC = \sD = \bF_p\mod$ be the categories of $\bF_p$-vector
spaces for $p$ some prime. We will construct a non-zero functor
$\bF_p\mod^{\leq 0} \to \bF_p\mod^{\leq 0}$ which
satisfies the hypotheses of Theorem \ref{t:diff'l}.

For $V \in \bF_p\mod^{\heart}$, let
$\Sym^{\on{perf}}(V)$ denote the perfection of the symmetric algebra
on $V$, i.e., $\colim_n \Sym(V)$ with Frobenius as structure
maps. By \cite{higheralgebra} Theorem 1.3.3.8, there is a unique
functor $\bL\Sym^{\on{perf}}(-):\bF_p\mod^{\leq 0} \to \bF_p\mod^{\leq 0}$
commuting with sifted colimits and whose restriction to
$\bF_p\mod^{\heart}$ is $\Sym^{\on{perf}}(-)$. 

By \cite{bhatt-scholze-grg} Proposition 11.6,
$\bL\Sym^{\on{perf}}(-)$ vanishes on $\bF_p\mod^{\leq -1}$;
in particular, its Goodwillie derivative vanishes.
Then the evident formula:

\[
\bL\Sym^{\on{perf}}(V \oplus W) = 
\bL\Sym^{\on{perf}}(V) \otimes \bL\Sym^{\on{perf}}(W)
\]

\noindent implies that for every 
$V \in \bF_p\mod^{\leq 0}$, the functor
$\big(W \mapsto \bL\Sym^{\on{perf}}(V \oplus W)\big)$
also has vanishing Goodwillie derivative.

\end{rem}

\subsection{The bilinear obstruction to linearity}\label{ss:bilinear}

To prove Theorem \ref{t:diff'l},
it is convenient to use the following
construction. 

In the notation of \emph{loc. cit}., note that $\psi$
commutes with all colimits if and only
if it commutes with pairwise direct
sums (because $\psi$ is reduced and commutes with sifted colimits). 
Therefore, $B_{\psi}$ may be understood as the obstruction
to $\psi$ commuting with all colimits.

For $\sF,\sG \in \sC^{\leq 0}$,
define:

\[
B_{\psi}(\sF,\sG) = 
\Coker(\psi(\sF)\oplus \psi(\sG) \to 
\psi(\sF \oplus \sG)) \in \sD.
\]

\noindent Note that:

\[
\psi(\sF \oplus \sG) = 
\psi(\sF) \oplus \psi(\sG) \oplus
B_{\psi}(\sF,\sG)
\]

\noindent because
the composition:

\[
\psi(\sF) \oplus \psi(\sG) \to 
\psi(\sF \oplus \sG) =
\psi(\sF \times \sG) \to 
\psi(\sF) \times \psi(\sG) =
\psi(\sF) \oplus \psi(\sG)
\]

\noindent is the identity.

Therefore, $\psi$ commutes with colimits
if and only if $B_{\psi} = 0$.

\subsection{Simplicial review}

We will prove Theorem \ref{t:diff'l} 
using some standard simplicial
methods.

Suppose $\sF_{\dot}$ is a simplicial
object in $\sC$. 
We let $|\sF_{\dot}|$ denote the
\emph{geometric realization} of this
simplicial diagram, i.e., the colimit.

Similarly, let 
$|\sF_{\dot}|_{\leq n}$ denote the partial geometric
realization
$\underset{\bDelta_{\leq n}^{op}} {\colim} \,
\sF_{\dot}$. Here $\bDelta_{\leq n} \subset \bDelta$ is 
the full subcategory of simplices
of order $\leq n$. 
Note that:

\[
|\sF_{\dot}| = \underset{n}{\colim} \, 
|\sF_{\dot}|_{\leq n}.
\]

Finally, we recall (c.f. to the proof of the Dold-Kan theorem):

\begin{lem}\label{l:simp}

For $n \geq 0$,
$\Coker(|\sF_{\dot}|_{\leq n} \to
|\sF_{\dot}|_{\leq n+1})$
is a direct summand of
$\sF_{n+1}[n+1]$.

\end{lem}

\subsection{}

Our main technique is the following.

\begin{lem}\label{l:filt}

Suppose $\psi:\sC^{\leq 0} \to \sD$ commutes with sifted colimits.
For every $\sF \in \sC$, $\psi(\Sigma \sF)$ admits an increasing filtration\footnote{For us, all filtrations are assumed exhaustive.}
$\fil_{\dot} \psi(\Sigma \sF)$ such that:

\begin{itemize}

\item $\fil_i \psi(\Sigma\sF) = 0$ for $i < 0$.

\item For $i \geq 0$, $\gr_i \psi(\Sigma \sF)$ is a direct summand
of $\psi(\sF^{\oplus i})[i]$. 

\item More precisely, $\gr_0 \psi(\Sigma \sF) = \psi(0)$,
$\gr_1(\psi(\Sigma \sF)) = \psi_{red}(\sF)[1]$, and 
$\gr_2(\psi(\Sigma \sF)) = B_{\psi_{red}}(\sF,\sF)[2]$.

\end{itemize}

\end{lem}

\begin{proof}

There is a canonical simplicial diagram:

\[
\ldots \sF\oplus \sF \rightrightrightarrows \sF \rightrightarrows 0
\]

\noindent with geometric realization $\Sigma \sF$. (For example, this
simplicial diagram is the Cech construction for $0 \to \Sigma \sF$.)

Because $\psi$ commutes with geometric realizations,
we have:

\[
\psi(\Sigma \sF) = |\psi(\sF^{\oplus \dot})|.
\]

\noindent We then set 
$\fil_i \psi(\Sigma \sF) = |\psi(\sF^{\oplus \dot})|_{\leq i}$.
This filtration tautologically satisfies the first property,
and it satisfies the second property by Lemma \ref{l:simp}.
The third property follows by refining Lemma \ref{l:simp}
to identify exactly which summand occurs, which we omit here
(see e.g. \cite{higheralgebra} Lemma 1.2.4.17).

\end{proof}

\begin{rem}\label{r:calc}

In the terminology of Goodwillie calculus, we generally have:

\[
\gr_i \psi(\Sigma \sF) = \on{cr}_i\psi(\sF,\ldots,\sF)[i]
\]

\noindent where $\on{cr}_i\psi$ is the \emph{$i$th cross-product} of $\psi$.
(In particular, $B_{\psi}$ is non-standard notation for $\on{cr}_2\psi$.)
However, we will not explicitly need the higher cross products.

\end{rem}

As a first consequence, we obtain:

\begin{cor}\label{c:tstr}

Suppose that $\psi:\sC^{\leq 0} \to \sD^{\leq 0}$
is reduced and commutes with sifted colimits.
Then for every $n \geq 0$, $\psi(\sC^{\leq -n}) \subset \sD^{\leq -n}$.

\end{cor}

\begin{proof}

By induction, 
it suffices to show $\psi(\sC^{\leq -1}) \subset \sD^{\leq -1}$. 
Suppose $\sF \in \sC^{\leq 0}$; 
we need to show $\psi(\Sigma \sF) \in \sD^{\leq -1}$.

We use the filtration of Lemma \ref{l:filt}. It suffices to show
$\gr_i \psi(\Sigma \sF) \in \sD^{\leq -1}$. Then $\gr_0 \psi(\Sigma \sF) = 0$
as $\psi$ is reduced; and for $i>0$, $\gr_i \psi(\Sigma \sF)$ is a summand
of $\psi(\sF^i)[i]$, which is clearly in degrees $\leq -i<0$.

\end{proof}

\subsection{}

We now prove Theorem \ref{t:diff'l}.

\begin{proof}[Proof of Theorem \ref{t:diff'l}]

We will show by induction on $n$ that
these hypotheses on $\psi$ force $\psi(\sC^{\leq -1}) \subset \sD^{\leq -n}$.
The case $n = 1$ is given by Corollary \ref{c:tstr}.
In what follows, we assume the inductive hypothesis
for $n$ and deduce it for $n+1$.

\step

First, we claim that for $\sF \in \sC^{\leq -1}$, the functor:

\[
B_{\psi}(\sF,-)[-1]:\sC^{\leq 0} \to \sD
\]

\noindent satisfies the hypotheses of Theorem \ref{t:diff'l}.

Clearly this functor is reduced and commutes with sifted colimits.

Let us show that this functor maps $\sC^{\leq 0}$ into $\sD^{\leq 0}$.
Fix $\sG \in \sC^{\leq 0}$. The reduced functor
$B_{\psi}(-,\sG)$ commutes with sifted colimits
and maps $\sC^{\leq 0}$ into $\sD^{\leq 0}$.
By Corollary \ref{c:tstr}, $B_{\psi}(\sF,\sG) \in \sD^{\leq -1}$,
so $B_{\psi}(\sF,\sG)[-1] \in \sD^{\leq 0}$ as desired.

Finally, note that for any $\sF^{\prime} \in \sC^{\leq -1}$, 
the functor:

\[
\Omega B_{\psi}(\sF,\sF^{\prime} \oplus -):\sC^{\leq 0} \to \sD
\]

\noindent has vanishing Goodwillie derivative,
as it is a summand of the functor
$\psi(\sF \oplus \sF^{\prime} \oplus -)[-1]$.

Therefore, we may apply the inductive hypothesis
to this functor. We obtain that $B_{\psi}(\sF,-)$
maps $\sC^{\leq -1}$ to $\sD^{\leq -n-1}$. 
In particular, $B_{\psi}(\sF,\sF) \in \sD^{\leq -n-1}$.

\step 

Next, we claim that:

\[
\Coker(\Sigma\psi(\sF) \to \psi(\Sigma \sF)) \in \sD^{\leq -n-3}.
\]

\noindent Note that by the construction of Lemma \ref{l:filt}, the
map:

\[
\Sigma \psi(\sF) = \gr_1 \psi(\Sigma \sF) =
\fil_1 \psi(\Sigma \sF) \to 
\psi(\Sigma \sF)
\]

\noindent is the canonical map used in the definition of the 
Goodwillie derivative. Therefore, it suffices to show that
$\gr_i \psi(\Sigma \sF) \in \sD^{\leq -n-3}$ for $i \geq 2$.

By induction,
$\psi(\sF^i) \in \sD^{\leq -n}$ by induction. Therefore,
$\gr_i \psi(\Sigma \sF) \in \sD^{\leq -n-i}$.
This gives the claim for $i \geq 3$.

If $i = 2$, then $\gr_i \psi(\Sigma \sF) = B_{\psi}(\sF,\sF)[2]$,
and by the previous step $B_{\psi}(\sF,\sF) \in \sD^{\leq -n-1}$ as needed.

\step 

Finally, the previous step implies that for $\sF \in \sC^{\leq -1}$,
the map:

\[
\psi(\sF) \to \Omega\psi(\Sigma \sF)
\]

\noindent is an isomorphism on $H^{-n}$ (where this notation denotes the 
cohomology functor for the $t$-structure on $\sD$).

More generally, for any $m \geq 0$, the functor
$\Omega^m \psi \Sigma^m: \sC^{\leq 0} \to \sD^{\leq 0}$ satisfies
our hypotheses, so we find that
$\Omega^m\psi\Sigma^m(\sF) \to \Omega^{m+1}\psi(\Sigma^{m+1} \sF)$
is an isomorphism on $H^{-n}$.

Finally, we obtain $H^{-n}(\psi(\sF)) \isom H^{-n}(\partial\psi(\sF))$ 
is an isomorphism. But of course, $\partial\psi = 0$, so we obtain
$\psi(\sF) \in \sD^{\leq -n-1}$, providing the inductive step.

\end{proof}

\subsection{Full vanishing results}\label{ss:extensible}

We now wish to extend the above results to give
vanishing for connective objects.
Throughout this section, $\sC,\sD\in\StCat_{cont}$ are equipped with
$t$-structures compatible with filtered colimits.

The following definition is obviously
quite natural in this context.

\begin{defin}

A functor $\psi: \sC^{\leq 0} \to \sD^{\leq 0}$ 
is \emph{extensible} if
there exists $\widetilde{\psi}:\sC^{\leq 1} \to \sD^{\leq 1}$ 
commuting with sifted colimits with 
$\widetilde{\psi}|_{\sC^{\leq 0}} = \psi$.

\end{defin}

It is convenient to introduce the following notion 
as well.

\begin{defin}

A functor $\psi:\sC^{\leq 0} \to \sD^{\leq 0}$ is \emph{pseudo-extensible} if:

\begin{itemize}

\item $\psi$ is reduced and
commutes with sifted colimits.

\item Every $\vph$ in $\sS_{\psi}$ maps
$\sC^{\leq 0} \to \sD^{\leq 0}$,
where $\sS_{\psi} \subset \TwoHom(\sC,\sD)$ 
is the minimal subgroupoid such 
that $\psi \in \sS_{\psi}$ and such that for every
$\sF \in \sC^{\leq 0}$ and $\vph \in \sS_{\psi}$,
$B_{\vph}(\sF,-)[-1] \in \sS_{\psi}$.

\end{itemize}

\end{defin}

\begin{rem}\label{r:analytic}

In standard terminology from Goodwillie calculus, 
one can show that if $\psi$ commutes with
sifted colimits, then $\psi$ is \emph{(-1)-analytic}
if and only if $\psi[n]$ is pseudo-extensible for some $n$.

\end{rem}

\begin{lem}\label{l:extn}

If $\psi:\sC^{\leq 0} \to \sD^{\leq 0}$ is extensible, then
it is pseudo-extensible.

\end{lem}

\begin{proof}

Clearly any $\vph \in \sS_{\psi}$ is extensible. Therefore,
by induction we are reduced to showing
that for $\sF \in \sC^{\leq 0}$, $B_{\psi}(\sF,-)[-1]$
maps $\sC^{\leq 0}$ to $\sD^{\leq 0}$. 

Let $\widetilde{\psi}:\sC^{\leq 1} \to \sD^{\leq 1}$
be as in the definition of extensibility.
Clearly $B_{\widetilde{\psi}}$ maps
$\sC^{\leq 1} \times \sC^{\leq 1}$ to $\sD^{\leq 1}$,
so applying Corollary \ref{c:tstr} once in 
each variable, we find
$B_{\widetilde{\psi}}$
maps $\sC^{\leq 0} \times \sC^{\leq 0}$ to $\sD^{\leq -1}$.
Here the functor coincides with $B_{\psi}$, and incorporating
the shift we get the claim.

\end{proof}

The flexibility the next result affords is ultimately the reason we
consider pseudo-extensible functors here.

\begin{lem}\label{l:ext'n-left-complete}

In the above setting, suppose the $t$-structure on 
$\sD$ is left complete. 
Let $\sI$ be a filtered category,
and suppose we are given a diagram $\sI^{op} \xar{i \mapsto \psi_i} 
\TwoHom(\sC^{\leq 0},\sD^{\leq 0})$ of pseudo-extensible
functors. Suppose that for every $n$ there exists
$i \in \sI$ such that
$\tau^{\geq -n} \psi_j \isom \tau^{\geq -n} \psi_i$ for all $i \to j \in \sI$.

Then the value-wise limit of functors 
$\psi = \lim_{i \in \sI^{op}} \psi_i$ is pseudo-extensible.

\end{lem}

\begin{proof}

For $\sF,\sG \in \sC^{\leq 0}$, we have:

\[
B_{\psi}(\sF,\sG) = \underset{i}{\lim} B_{\psi_i}(\sF,\sG)
\]

\noindent by definition of $B_{-}$. 
Therefore, so we are reduced (by induction, say) to showing
that $B_{\psi}(\sF,\sG) \in \sD^{\leq -1}$.

Recall that left completeness and our stabilization
hypotheses imply $\tau^{\geq -n} \psi =
\tau^{\geq -n} \psi_i$ for $i$ in $\sI$ sufficiently large
(depending on $n$). Therefore, 
$\tau^{\geq -n} \psi$ is reduced and commutes with sifted
colimits for every $n$, which implies the same for the functor $\psi$
(by left completeness of the $t$-structure on $\sD$).

Similarly, we have $\tau^{\geq -n} \underset{i}{\lim} B_{\psi_i}(\sF,\sG) =
\tau^{\geq -n} B_{\psi_i}(\sF,\sG)$ for $i$ sufficiently
large. So clearly $B_{\psi}(\sF,\sG) \in \sD^{\leq -1}$,
since this is true for each $\psi_i$ by pseudo-extensibility.

\end{proof}

We now have the following result, 
which in the extensible case is 
just a rephrasing of Theorem \ref{t:diff'l}.

\begin{thm}\label{t:full-vanishing}

In the setting of Theorem \ref{t:diff'l},
suppose $\psi$ is pseudo-extensible and
${\partial(\psi(\sF \oplus -))} = 0$
for every $\sF \in \sC^{\leq 0}$.
Then $\psi$ maps $\sC^{\leq 0}$ to $\cap \sD^{\leq -n}$. 

\end{thm}

\begin{proof}

First, we claim $\psi(\sF) \in \sD^{\leq -1}$ for all 
$\sF \in \sC^{\leq 0}$. As in the proof of
Theorem \ref{t:diff'l}, it suffices to show
$\gr_i\psi(\Sigma \sF) \in \sD^{\leq -3}$ for all $i$,
and this is automatic for $i \geq 3$. For $i = 2$,
we have $\gr_2\psi(\Sigma \sF) = B_{\psi}(\sF,\sF)[2]$,
and $B_{\psi}(\sF,\sF) \in \sD^{\leq -1}$ by pseudo-extensibility.

Next, observe that any 
$\vph \in \sS_{\psi}$ is pseudo-extensible, and by
induction the Goodwillie derivatives of the functors
$\vph(\sF\oplus -)$ vanish for any $\sF \in \sC^{\leq 0}$.
Therefore, by the above argument, every $\vph \in \sS_{\psi}$
maps $\sC^{\leq 0}$ into $\sD^{\leq -1}$.

Finally, we see that $\psi[-1]$ is pseudo-extensible, so 
by induction we obtain the result.

\end{proof}

We immediately deduce the following.

\begin{cor}\label{c:full-vanishing}

In the setting of Theorem \ref{t:conn-algs},
suppose that the functors $\Psi_{A,red}:A\bimod^{\leq 0} \to \Sp$ 
are pseudo-extensible.

Then $\Psi(A\oplus M) \isom \Psi(A)$ for any $M \in A\bimod^{\leq 0}$.

\end{cor}

\subsection{The split square-zero case of Theorem \ref{t:dgm}}\label{ss:split}

The following two results will be shown
in \S \ref{s:k} and \S \ref{s:tc} respectively.

\begin{thm}\label{t:k}

\begin{enumerate}

\item\label{i:sifted}

For $A \in \Alg_{conn}$, the functor:

\[
\begin{gathered} 
A\bimod^{\leq 0} \to \Sp^{\leq 0} \\
M \mapsto K(\SqZero(A,M))
\end{gathered}
\]

\noindent commutes with sifted colimits.
Moreover, the underlying reduced functor is 
extensible in the sense of \S \ref{ss:extensible}.

\item\label{i:trace} For $A \in \Alg_{conn}$, the Goodwillie derivative
of the functor:

\[
\begin{gathered}
A\bimod^{\leq 0} \to \Sp \\
M \mapsto K(A \oplus M)
\end{gathered}
\]

\noindent is canonically isomorphic to the functor $M \mapsto \THH(A,M)[1]$.

\end{enumerate}

 \end{thm}

\begin{thm}\label{t:tc}

\begin{enumerate}

\item\label{i:ps-ext}

For $A \in \Alg_{conn}$, the functor:

\[
\begin{gathered}
A\bimod^{\leq 0} \to \Sp \\
M \mapsto \TC(\SqZero(A,M))
\end{gathered}
\]

\noindent is pseudo-extensible in the sense of 
\S \ref{ss:extensible}.

\item\label{i:tc-conn}

For $A$ as above and $M \in A\bimod^{\leq 0}$:
 
\[
\TC_{red}(\SqZero(A,M)) \coloneqq \Ker(\TC(\SqZero(A,M)) \to \TC(A))
\in \Sp^{\leq -1}.
\]

\item\label{i:tc-deriv}

The Goodwillie derivative of the
above functor is canonically isomorphic to $\THH(A,-)[1]$.
Moreover, this isomorphism is compatible with the
cyclotomic trace and the isomorphism of Theorem \ref{t:k}.

\end{enumerate}

\end{thm}

\begin{rem}\label{r:k-conn}

Throughout these notes, $K$ denotes \emph{connective} $K$-theory.
Adapting \cite{beilinson} Lemma 2.3 to the setting of 
connective ring spectra justifies omitting 
negative $K$-groups.

\end{rem}

\subsection{}

Let us assume the above theorems for now and deduce 
Theorem \ref{t:dgm} in the split square-zero case.

Define the functor:

\[
\begin{gathered}
\Psi^{DGM}: \Alg_{conn} \to \Sp \\
A \mapsto \Coker(K(A) \to \TC(A))
\end{gathered}
\]

\noindent where the map $K(A) \to \TC(A)$ is the
cyclotomic trace map. 
It suffices to show $\Psi^{DGM}$ satisfies
the hypotheses of Theorem \ref{t:conn-algs} and
Corollary \ref{c:full-vanishing}.

$\Psi_{red}^{DGM}(\SqZero(A,-)):A\bimod^{\leq 0} \to \Sp$ is
pseudo-extensible by Theorem \ref{t:k} \eqref{i:sifted}
and Theorem \ref{t:tc} \eqref{i:ps-ext}
as this property is preserved under cokernels.
Moreover, this functor maps into $\Sp^{\leq 0}$
by Theorem \ref{t:tc} \eqref{i:tc-conn} (c.f. Remark \ref{r:k-conn}).
Finally, its Goodwillie derivative vanishes
by Theorem \ref{t:k} \eqref{i:trace} and 
Theorem \ref{t:tc} \eqref{i:tc-deriv}.

\section{K-theory}\label{s:k}

\subsection{}

In this section, we prove Theorem \ref{t:k}.
It is convenient in working with $K$-theory
to generalize to a categorical setting, and we do so
in what follows.

\subsection{Split square-zero extensions categorically}

First, we interpret the theory of 
split square-zero extensions in the categorical 
setting.

Suppose $\sC \in \StCat_{cont}$ and $T:\sC \to \sC \in \StCat_{cont}$ 
is a (continuous, exact) endomorphism.

\begin{defin}

$\SqZero(\sC,T)$ is the category of pairs $\sF \in \sC$ and $\eta:\sF \to T(\sF)$ a 
\emph{locally nilpotent} endomorphism, i.e., the colimit of the diagram:

\[
\sF \xar{\eta} T(\sF) \xar{T(\eta)} T(\sF) \xar{T^2(\eta)} \ldots 
\]

\noindent is zero.

\end{defin}

\begin{prop}\label{p:sqzero-comparison}

Suppose $A \in \Alg$ and $M \in A\bimod$. Let $T_M \coloneqq (M[1] \otimes_A -):
A \mod \to A\mod$. Then there is a canonical equivalence:

\[
\SqZero(A,M)\mod \simeq \SqZero(A\mod,T_M)
\]

\noindent such that the diagram:

\[
\xymatrix{
\SqZero(A,M)\mod \ar[rr] \ar[dr]_{A \underset{\SqZero(A,M)}{\otimes} -} 
&& \SqZero(A\mod,T_M) \ar[dl]^{(\sF,\eta) \mapsto \sF} \\
& A\mod
}
\]

\noindent commutes.

\end{prop}

\begin{proof}

We construct the functor:

\[
F:\SqZero(A,M)\mod \to \SqZero(A\mod,T_M)
\]

\noindent as follows. Suppose $N \in \SqZero(A,M)\mod$.
We must have $F(N) = A \otimes_{\SqZero(A,M)} N$ as objects of
$A\mod$; it remains to define the map $\eta$ (in the above notation).
This map is the boundary for the obvious exact triangle:

\[
T_M(F(N))[-1] = M \underset{A}{\otimes} F(N) \to N \to F(N) \xar{+1}
\]

\noindent where $N$ is regarded as an $A$-module through the section
$A \to \SqZero(A,M)$. In the case $N = \SqZero(A,M)$, one
sees that this triangle is split, so $\eta$ is $0$; this implies
in general that $\eta$ is locally nilpotent.

Now observe that the diagram:

\[
\xymatrix{
\SqZero(A,M)\mod \ar[rr]^F \ar[dr] & & 
\SqZero(A\mod,T_M) \ar[dl]^{(\sF,\eta) \mapsto \Ker(\eta)} \\
& A\mod
}
\]

\noindent tautologically commutes, where the left arrow
is restriction along the morphism $A \to \SqZero(A,M)$.

To show that $F$ is an equivalence, it 
suffices to show that each of the above functors to $A\mod$ is monadic
and the induced functor of monads is an isomorphism.
The left functor is tautologically monadic. The right functor
admits the left adjoint $\sF \mapsto (\sF, \eta = 0)$;
it obviously commutes with colimits and is conservative
by local nilpotence of $\eta$. Recall that to check the induced
map of monads is an isomorphism, it is enough to see that 
$F$ is intertwined by the left adjoints to the vertical arrows,
which is evident. 

Moreover, this equivalence also clearly makes the diagram
from the proposition commute.

\end{proof}

We also use the following observation.

\begin{prop}\label{p:sqzero-cpt}

In the above setting, suppose that $\sC$ is compactly generated.
Then $\SqZero(\sC,T)$ is compactly generated. An object
of $\SqZero(\sC,T)$ is compact if and only if the underlying
object of $\sC$ is. 

\end{prop}

\begin{proof}

Suppose $(\sF,\eta_{\sF}), (\sG,\eta_{\sG}) \in \SqZero(\sC,T)$.
Note that:

\[
\ul{\Hom}_{\SqZero(\sC,T)}\big((\sF,\eta_{\sF}), (\sG,\eta_{\sG})\big) =
\on{Eq} \big(\ul{\Hom}_{\SqZero(\sC,T)}(\sF,\sG) \rightrightarrows
\ul{\Hom}_{\SqZero(\sC,T)}(\sF,\sG)\big)
\]

\noindent where the two maps in the equalizer 
are composition with $\eta_{\sF}$ and
$\eta_{\sG}$ respectively. This implies that if $\sF$
is compact in $\sC$, then $(\sF,\eta_{\sF})$ is compact
in $\SqZero(\sC,T)$ (as colimits commutes with finite limits in $\sC$).

We claim that $\SqZero(\sC,T)$ is compactly generated by
objects $(\sF,\eta = 0)$ for $\sF$ compact in $\sC$.
Indeed, suppose $(\sG,\eta_{\sG}) \in \SqZero(\sC,T)$ receives
only the zero map from such objects. This implies $\Ker(\eta_{\sG}) = 0$,
as:

\[
\ul{\Hom}_{\SqZero(\sC,T)}\big((\sF,0), (\sG,\eta_{\sG})\big) =
\ul{\Hom}_{\sC}(\sF,\Ker(\eta_{\sG})).
\]

\noindent Then local nilpotence of $\eta_{\sG}$ implies
$\sG = 0$, as desired.

\end{proof}

\subsection{Variant}\label{ss:ct}

Here is a sort of alternative to the square-zero extension construction above,
which is more convenient for our purposes.

Let $\sC$ be a compactly generated
stable category and let $T:\sC \to \sC \in \StCat_{cont}$ be an endomorphism.
We define $\sC^T$ as $\Ind(\sC^{T,c})$ where
$\sC^{T,c}$ is the category of pairs $(\sF,\eta)$ with
$\sF \in \sC^c$ and $\eta:\sF \to T(\sF)$. We remark that there
is no local nilpotence hypothesis here. 

\begin{example}\label{e:t=id}

For $k$ a field and $\sC = k\mod$ and $T = \id$, there
are natural identifications:

\[
\begin{gathered}
\SqZero(\sC,T) = 
\Ker(k[t]\mod \xar{k(t) \underset{k[t]}{\otimes} -} k(t)\mod) \\
\sC^T = 
\Ker(k[t]\mod \xar{k[t,t^{-1}] \underset{k[t]}{\otimes} -} k[t,t^{-1}]\mod).
\end{gathered}
\]

\end{example}

\subsection{}

We have the following compatibility in the above setting.
Note that there is always a fully-faithful functor
$\SqZero(\sC,T) \to \sC^T$ preserving compact objects.

\begin{lem}\label{l:ct=sqzero}

For connective $A$ and $M \in A\bimod^{\leq 0}$, 
the natural functor:

\[
\SqZero(A\mod,T_M) \to A\mod^{T_M}
\]

\noindent is an equivalence of categories.

In particular, $A\mod^{T_M}$ is canonically equivalent
to $\SqZero(A,M)\mod$.

\end{lem}

\begin{proof}

It suffices to observe that for any $\sF \in A\mod^c$,
any map $\eta:\sF \to T_M(\sF)$ is automatically
locally nilpotent as $\colim_n T_M^n(\sF) = 0$
(since $\sF$ is bounded above and $T_M = M \otimes_A -[1]$ 
lowers cohomological degrees by $1$).

The second point follows from 
Proposition \ref{p:sqzero-comparison}.

\end{proof}

\begin{rem}

Note that there is no contradiction with
Example \ref{e:t=id}: there $\id_{\sC} = T_{k[-1]}$
and $k[-1]$ is not connective.

\end{rem}

\subsection{Structural features of $K$-theory}

We now prove the first point of Theorem \ref{t:k}.

\begin{proof}[Proof of Theorem \ref{t:k} \eqref{i:sifted}]

To prove the extensibility, define:

\[
\begin{gathered}
\widetilde{K}_A: A\bimod^{\leq 1} \to \Sp^{\leq 0} \subset \Sp \\
\widetilde{K}_A(M) \coloneqq K(A\mod^{T_M,c}).
\end{gathered}
\]

For $M \in A\bimod^{\leq 0}$, note that
$\widetilde{K}_A(M) = K(\SqZero(A,M))$ by Lemma \ref{l:ct=sqzero}.
Therefore, we need only to show that $\widetilde{K}_A$ commutes with
sifted colimits. 

Let $\Proj(A) \subset A\mod^c$ denote the full subcategory
of \emph{(finitely-generated) projective} $A$-modules,
i.e., the full subcategory of $A\mod^c$ consisting of summands of $A^{\oplus n}$ 
for some $n \in \bZ^{\geq 0}$. 
Let $\Proj(A)^{T_M} \subset A\mod^{T_M,c}$ be the full
subcategory of pairs $(\sF,\eta)$ with $\sF \in \Proj(A)$.

Note that $\Proj(A)^{T_M}$ is an exact\footnote{In the
higher categorical sense: see \cite{barwick-exact} for an introduction
in this setup.} category. 
Standard\footnote{See \cite{fontes} for a general format
for such problems. In particular, the main theorem of \emph{loc. cit}.
implies our present claim. We thank Thomas Nikolaus for directing us
to this reference.} arguments show that
$K(\Proj(A)^{T_M}) \isom \widetilde{K}_A(M)$,
where the left hand side indicates Waldhausen $K$-theory of this
exact category. (It is essential that $T_M$ is right $t$-exact
here.)

Because $\Omega^{\infty}: \Sp^{\leq 0} \to \Gpd$ commutes with sifted
colimits (as follows from thinking of connective spectra as
group-like $\sE_{\infty}$-monoids), it suffices 
to show that the functor:

\[
\Omega^{\infty}\widetilde{K}_A: A\bimod^{\leq 1} \to \Gpd
\]

\noindent commutes with sifted colimits, or just as well,
that $\Omega^{\infty-1}\widetilde{K}_A$ does.
The latter is the geometric realization of Waldhausen's 
$S_{\dot}$ construction, so it suffices to show the individual
terms of the $S_{\dot}$ construction commute with sifted
colimits in $M$ here.

To simplify the notation, we explain why
$M\mapsto S_2(\Proj(A)^{T_M}) \in \Gpd$ commutes
with sifted colimits: the general case is the
same with more notation.

First, recall that $S_2(\Proj(A))$ is the
groupoid of data $\sF_1,\sF_2 \in \Proj(A)$
and a map $f:\sF_1 \to \sF_2$ 
with $\Coker(f) \in \Proj(A) \subset A\mod$.

Then we similarly have:

\[
S_2(\Proj(A)^{T_M}) = 
\underset{(f:\sF_1 \to \sF_2) \in S_2(\Proj(A))} {\colim} \,
\Hom_{A\mod}(\sF_1,T_M(\sF_1)) 
\underset{\Hom_{A\mod}(\sF_1,T_M(\sF_2)) }{\times} 
\Hom_{A\mod}(\sF_2,T_M(\sF_2)).
\]

\noindent Therefore, it suffices 
to show that for every point
$(f:\sF_1 \to \sF_2) \in S_2(\Proj(A))$,
the above expression commutes with sifted colimits in $M$.

Note that:

\[
\begin{gathered}
\Hom_{A\mod}(\sF_1,T_M(\sF_1)) 
\underset{\Hom_{A\mod}(\sF_1,T_M(\sF_2)) }{\times} 
\Hom_{A\mod}(\sF_2,T_M(\sF_2)) = \\
\Omega^{\infty} \Big(
\ul{\Hom}_{A\mod}(\sF_1,T_M(\sF_1)) 
\underset{\ul{\Hom}_{A\mod}(\sF_1,T_M(\sF_2)) }{\times} 
\ul{\Hom}_{A\mod}(\sF_2,T_M(\sF_2)) 
\Big).
\end{gathered}
\]

\noindent Before passing to $\Omega^{\infty}$,
this expression clearly commutes with all colimits 
in $M$: this follows from compactness of the $\sF_i$
and the fact that $A\mod$ is stable.

As $\Omega^{\infty}:\Sp^{\leq 0} \to \Gpd$ 
is conservative and commutes
with sifted colimits, it suffices to show that
this fiber product lies in $\Sp^{\leq 0}$.
Each term in the fiber product is connective
because $\sF_i \in \Proj(A)$ and
$T_M(\sF_i) \in A\mod^{\leq 0}$. Then the
fiber product is connective because
$\ul{\Hom}(\sF_2,T(\sF_2)) \to 
\ul{\Hom}(\sF_1,T(\sF_2))$ is surjective
on $H^0$ (because $\Coker(f) \in \Proj(A)$). 

\end{proof}

\subsection{Derivative of $K$-theory}

We now give the calculation of Goodwillie derivatives in
an appropriate categorical setup.

\subsection{}

For $(\sC,T)$ as in \S \ref{ss:ct},
let $K_{\sC}(T) \in \Sp$ be the (connective) $K$-theory
of $\sC^{T,c}$. 

We have the following easy result.

\begin{lem}\label{l:kc-filt}

For every compactly generated $\sC$, 
the functor $K_{\sC}:\TwoEnd_{\StCat_{cont}}(\sC) \to \Sp$
commutes with filtered colimits.

\end{lem}

\begin{proof}

Clearly the functor
$T \mapsto \sC^{T,c}$ (as a functor to essentially small stable categories, say)
commutes with filtered colimits, so the result follows from the
commutation of $K$-theory and filtered colimits. 

\end{proof}

\subsection{}\label{ss:lax-func}

Note that $K_{\sC}$ has a little more functoriality.

Let $\StCat_{cg,endo}$ be the following 1-category.\footnote{Meaning
$(\infty,1)$-category, of course; we are distinguishing
here from an $(\infty,2)$-category.} 
Objects are pairs $(\sC,T)$ with 
$\sC$ a compactly generated stable category
and $T:\sC\to \sC$ a functor commuting with colimits.
Morphisms are lax commuting diagrams:

\[
\xymatrix{
\sC \ar[rr]^{T} \ar[d]^F & & \sC \ar[d]^F \ar@{=>}[dll]_{\vareps} \\
\sD \ar[rr]^{T^{\prime}} \ar[rr]  & & \sD
}
\]

\noindent (so $\vareps:F T \to T^{\prime} F$ is a natural
transformation) with $F$ preserving compact 
objects.\footnote{This does not completely define a structure
of category, of course: we have not
written compositions, never mind higher data. 
But all of this data is implicit in the standard $2$-categorical
structure on $\StCat_{cont}$; we refer to \cite{grbook} for
the appropriate formalism, including how to properly
define this category.}

Then $(\sC,T) \mapsto K_{\sC}(T)$ upgrades to a functor
out of $\StCat_{cg,endo}$. 
For a diagram as above, 
the map $K_{\sC}(T) \to K_{\sD}(T^{\prime})$
is induced by the functor:

\[
\begin{gathered}
\sC^{T,c} \to \sD^{T^{\prime},c} \\
(\sF,\eta:\sF \to T(\sF)) \mapsto 
\big(F(\sF),F(\sF) 
\xar{F(\eta)} FT(\sF) \xar{\vareps} T^{\prime}F(\sF)\big).
\end{gathered}
\]

\subsection{}

In the setting of \S \ref{ss:ct}, 
let $\partial K_{\sC}$ denote the Goodwillie derivative
of the functor $K_{\sC}:\TwoEnd_{\StCat_{cont}}(\sC) \to \Sp$;
although $K_{\sC}$ does not commute with sifted colimits,
the definition and construction of the 
Goodwillie derivative still apply
(c.f. \S \ref{ss:convention}).
Moreover, $\partial K_{\sC}$
commutes with arbitrary colimits by Lemma \ref{l:kc-filt}.

\subsection{}

By Lemma \ref{l:ct=sqzero}, the following result is
a generalization of Theorem \ref{t:k} \eqref{i:trace}.

\begin{thm}\label{t:k-deriv}

The functor $\partial K$ is canonically isomorphic to the
trace functor $\tr_{\sC}:\End(\sC) \to \Sp$. 

\end{thm}

The proof of this result occupies the remainder of this section.

\subsection{}

We will prove Theorem \ref{t:k-deriv} 
using the following convenient characterization 
of $\partial K$.

First, note that $K(\Sp)$ has a canonical 
base-point, i.e., there's a canonical
map $\bS \to K(\Sp) \in \Sp$ corresponding
to the point of $\Omega^{\infty} K(\Sp)$ which
is the class of the sphere spectrum. Similarly,
we have a canonical point of $K_{\Sp}(\id_{\Sp})$
defined by $(\bS,\id_{\bS}) \in \Sp^{\id_{\Sp},c}$.

We recall that the notation $\StCat_{cg,endo}$ introduced
in \S \ref{ss:lax-func}.

\begin{lem}\label{l:dk-init}

Suppose that we are given a functor
$\Phi: \StCat_{cg,endo} \to \Sp$, which we denote
by $(\sC,T) \mapsto \Phi_{\sC}(T)$. 
Suppose that we are given a base-point 
$x:\bS \to \Phi_{\Sp}(\id_{\Sp}) \in \Sp$.

Suppose moreover that:

\begin{itemize}

\item For every $\sC$, the functor $\Phi_{\sC}(-)$ 
commutes with colimits.

\item The functor $\Phi$ is \emph{additive} in the following
sense. Abuse notation in writing
$T$ for the endofunctor of
$\TwoHom(\Delta^1,\sC) = \{\sF \to \sG \in \sC\}$
sending $\sF \to \sG$ to $T(\sF) \to T(\sG)$.
Then we suppose (using the notation
of \S \ref{ss:ct}) that the map:

\[
\begin{gathered}
\TwoHom(\Delta^1,\sC)^T \to \sC \times \sC \\
\vcenter{\xymatrix{\sF \ar[d]^{\eta_{\sF}} \ar[r]^f & \sG \ar[d]^{\eta_{\sG}} \\ 
T(\sF) \ar[r] & T(\sG)} } \mapsto \big((\sF,\eta_{\sF}),(\Coker(f),\eta_{\Coker(f)})\big)
\end{gathered}
\]

\noindent induces an isomorphism 
$\Phi_{\TwoHom(\Delta^1,\sC)}(T) \isom 
\Phi_{\sC}(T) \times \Phi_{\sC}(T)$.

\end{itemize}

Then there is a unique natural transformation:

\[
\partial K_{\sC}(T) \to \Phi_{\sC}(T)
\]

\noindent of functors $\StCat_{cg,endo} \to \Sp$
equipped with a structure of based map when
evaluated on $(\Sp,\id_{\Sp})$. In particular,
$(\sC,T) \mapsto \partial K_{\sC}(T)$ is initial with respect 
to the above data.

\end{lem}

\begin{proof}

Suppose $(\sC,T) \in \StCat_{cg,endo}$.
By functoriality of $\Phi$ and using its base-point,
we have a canonical map:

\[
\Hom_{\StCat_{cg,endo}}((\Sp,\id_{\Sp}),(\sC,T)) =
\sC^{T,c,\simeq} \to 
\Hom_{\Sp}(\Phi_{\Sp}(\id_{\Sp}),\Phi_{\sC}(T)) \to 
\Omega^{\infty} \Phi_{\sC}(T) \in \Gpd.
\]

\noindent By additivity of $\Phi$ and the Waldhausen
construction of $K$-theory, this map factors through
a canonical map from $\Omega^{\infty}K_{\sC}(T)$,
and by Lemma \ref{l:nat-trans-sp} this
canonically upgrades to a map of spectra
$K_{\sC}(T) \to \Phi_{\sC}(T)$. Finally, by
definition of the Goodwillie derivative, 
this natural transformation factors through
$\partial K_{\sC}(T) \to \Phi_{\sC}(T)$. Clearly
this construction is natural in $(\sC,T)$, giving the claim.

\end{proof}

We used the following result in the course
of the proof, which we explicitly record for clarity.

\begin{lem}\label{l:nat-trans-sp}

For $\sD$ stable and $F,G: \sD \to \Sp$ exact
functors, natural transformations between $F$ and $G$ are the
same as natural transformations between the functors
$\Omega^{\infty}F,\Omega^{\infty}G: \sD \to \Gpd$.
That is, the natural map:

\[
\Hom_{\TwoHom(\sD,\Sp)}(F,G) \to 
\Hom_{\TwoHom(\sD,\Gpd)}(\Omega^{\infty}F,\Omega^{\infty} G)
\]

\noindent is an isomorphism.

\end{lem}

\begin{proof}

We have:

\[
\Hom_{\TwoHom(\sD,\Sp)}(F,G) = 
\underset{n}{\lim} \, 
\Hom_{\TwoHom(\sD,\Gpd)}\big(\Omega^{\infty}(F[n]),
\Omega^{\infty}(G[n])\big). 
\]

\noindent Now observe that each of the structural maps
in this limit is an isomorphism (as suspension is an equivalence
for both $\sC$ and $\sD$).

\end{proof}

\subsection{}

We now prove Theorem \ref{t:k-deriv}.

\begin{proof}[Proof of Theorem \ref{t:k-deriv}]

We verify that the functor $(\sC,T) \mapsto \tr_{\sC}(T)$
satisfies the same universal property as in Lemma \ref{l:dk-init}.

\step 

First, note that there actually is a canonical such a functor out
of $\StCat_{cg,endo}$: this follows from the
functoriality of traces discussed in \S \ref{s:tc} (and almost
established in \cite{kp} \S 1).

Clearly this functor commutes with colimits in $T$.
It is straightforward to check additivity;
we omit the verification here.

Moreover, $\tr_{\Sp}(\id_{\Sp}) = \bS \in \Sp$, so
there is a tautological base-point here.

It remains to show universality of the trace.
So suppose that we are given $\Phi$ as in 
Lemma \ref{l:dk-init}. 

\step\label{st:vareps}

We now make some preliminary constructions.

Fix $(\sC,T) \in \StCat_{cg,endo}$ and suppose
$\sF \in \sC^c$ compact. Then there is a canonical map:

\[
\vareps_{\sF}:\Hom_{\sC}(\sF,T(\sF)) \to 
\sC^{T,c,\simeq} \to \Omega^{\infty}\Phi_{\sC}(T) \in \Gpd
\]

\noindent with the first map being obvious and the
second map coming from Lemma \ref{l:dk-init}.

Moreover, suppose that for some $n \geq 0$, we are given a diagram: 

\[
\sF_0 \xar{\alpha_0} \ldots \xar{\alpha_{n-1}} \sF_n \xar{\alpha_n} 
T(\sF_0) \xar{T(\alpha_0)} \ldots \xar{T(\alpha_{n-1})} T(\sF_n) \in \sC
\]

\noindent with each $\sF_i$ compact. For each $i$, we have 
an induced map $\sF_i \to T(\sF_i)$, and we claim that
the induced point of $\Omega^{\infty}\Phi_{\sC}(T)$ is canonically
independent of $i$. More precisely, we have a simplicial
groupoid sending $[n]$ to the groupoid of diagrams
as above, and we claim there is a natural transformation
to the constant simplicial groupoid with value 
$\Omega^{\infty}\Phi_{\sC}(T)$ that coincides with
the construction $\vareps_{\sF}$ for $n = 0$.
(The additivity of $\Phi$ is essential here.) 

First, observe
that for any $\sF \in \sC^c$, 
$\vareps_{\sF}$ is pointed (and even upgrades to a map of
spectra) by exactness of $\Phi_{\sC}(-)$. Under the above
hypotheses, we may regard $\sF_n$ as a filtered object of
$\sC$, and the given data as a filtered map
$\sF_n \to T(\sF_n)$. Clearly on associated graded,
the induced maps:

\[
\Coker(\sF_i \to \sF_{i+1}) \to \Coker(T(\sF_i) \to T(\sF_{i+1}))
\]

\noindent are zero for all $i \geq 0$. Additivity
(and the pointedness noted above) then implies that:

\[
\vareps_{\sF_n}(T(\alpha_{n-1} \ldots T(\alpha_0) \alpha_n)) =
\vareps_{\sF_0}(\alpha_n \ldots \alpha_1\alpha_0).
\]
 
This argument immediately upgrades to give the 
desired natural transformation of simplicial groupoids; we omit
the details. Moreover, we note that these constructions
are natural in $(\sC,T) \in \StCat_{cg,endo}$ in the obvious sense.

\step 

Now fix $\sC$ compactly generated and stable.
Recall that $\sC$ is dualizable in $\StCat_{cont}$, so:

\[
\sC \otimes \sC^{\vee} \isom \TwoEnd_{\StCat_{cont}}(\sC).
\]

\noindent Therefore, it is enough to give the natural
transformation of the induced functors:

\[
\sC \times \sC^{\vee} \to \Sp.
\]

\noindent The left hand side is $\Ind(\sC^c \times \sC^{c,op})$,
so it is enough to construct our natural transformation 
when restricted to $\sC^c \times \sC^{c,op}$, i.e.,
for functors of the form $\sG \boxtimes \bD \sF$ for
$\sF,\sG \in \sC^c$. 

\step 

We have $\tr_{\sC}(\sG \boxtimes \bD \sF) = \ul{\Hom}_{\sC}(\sF,\sG)$.
By Lemma \ref{l:nat-trans-sp}, it suffices to
construct natural maps:

\[
\Hom_{\sC}(\sF,\sG) \to \Phi_{\sC}(\sG \boxtimes \bD \sF) \in \Gpd.
\]

We have map of spectra:

\[
\ul{\Hom}_{\sC}(\sF,\sG) \to 
\ul{\End}_{\sC}(\sF) \otimes \ul{\Hom}(\sF,\sG) \isom 
\ul{\Hom}(\sF,\ul{\End}_{\sC}(\sF) \otimes \sG) =
\ul{\Hom}(\sF,(\sG \boxtimes \bD \sF)(\sF)).
\]

\noindent Applying $\Omega^{\infty}$ and using the
construction from Step \ref{st:vareps}, we obtain a
map:

\[
\Hom_{\sC}(\sF,\sG) \to \Omega^{\infty}\Phi_{\sC}(\sG \boxtimes \bD \sF).
\]

\noindent This map is clearly natural in the variable $\sG$,
only natural with respect to isomorphisms in the variable $\sF$.
I.e., we have constructed a natural transformation
of functors $\sC^c \times \sC^{c,op,\simeq} \to \Gpd$.

\step It remains to upgrade the above construction to a natural transformation
of functors defined on all of $\sC^c \times \sC^{c,op}$.

First, at a homotopically naive level, 
suppose we are given $f:\sF_0 \to \sF_1$ and 
$g:\sF_1 \to \sG$. We a priori obtain two points
of $\Omega^{\infty}\Phi_{\sC}(\sG \boxtimes \bD \sF_0)$:

\begin{equation}\label{eq:func}
\begin{gathered}
gf \in \Hom_{\sC}(\sF_0,\sG), \hspace{.5cm} 
\Hom_{\sC}(\sF_0,\sG) \to \Omega^{\infty}\Phi(\sG \boxtimes \bD \sF_0) \\
f \in \Hom_{\sC}(\sF_1,\sG), \hspace{.5cm} 
\Hom_{\sC}(\sF_1,\sG) \to \Omega^{\infty}\Phi(\sG \boxtimes \bD \sF_1)
\Omega^{\infty}\Phi(\sG \boxtimes \bD \sF_0).
\end{gathered}
\end{equation}

\noindent We claim that they are canonically identified. 

We prove this by identifying both points with a third:
we have a canonical map:

\[
\sF_1 = \bS \otimes \sF_1 \xar{f\otimes g} \ul{\Hom}(\sF_0,\sF_1) \otimes \sG =
(\sG \boxtimes \bD \sF_0)(\sF_1)
\] 

\noindent which (by Step \ref{st:vareps}) gives a point of
$\Omega^{\infty} \Omega^{\infty}\Phi(\sG \boxtimes \bD \sF_0)$.

It is tautological that this map coincides with
the second map in \eqref{eq:func}. To identify it with
the first, we use the diagram:

\[
\sF_0 \xar{f} \sF_1 \xar{\id_{\sF_0} \otimes g}
(\sG \boxtimes \bD \sF_0)(\sF_0) \to 
(\sG \boxtimes \bD \sF_0)(\sF_1)
\]

\noindent and Step \ref{st:vareps}.

To upgrade this map to a homotopically correct one, one
shows that we have a morphism of the complete Segal spaces 
defined by $\sC^c \times \sC^{c,op}$ and $\Gpd$ respectively;
obviously this uses the full simplicial construction
from Step \ref{st:vareps}. We leave the details to the
reader.

\end{proof}

\section{Topological cyclic homology}\label{s:tc}

\subsection{}

In this section, we prove Theorem \ref{t:tc}.
We will deduce this result from an explicit calculation of $\TC$ 
for split square-zero extensions, see Theorem \ref{t:tc-calc} below.

\begin{rem}

Throughout this section, if not otherwise mentioned,
stable categories lie in $\StCat_{cont}$ and functors
between stable categories are morphisms there (i.e., continuous
exact functors).

\end{rem}

\subsection{Mea culpa and references}

Throughout this section, we need various functoriality properties
of traces. Unfortunately, these are not so well
documented at the moment.

In \S \ref{ss:trace-review}, \S \ref{ss:trace-review-2} and 
\S \ref{ss:tate-diag},
we indicate what functoriality we require. This material (especially
first two of these sections), is well-known folklore that does not
seem to quite have a convenient reference. 

We do not feel so much guilt on this point for three 
reasons. First, some of this
functoriality is established in \cite{kp} \S 1. Moreover, the
constructions (especially
Proposition 1.2.9) from \emph{loc. cit}. can be readily 
be generalized to provide the
desired functoriality using Segal spaces.

Second, Thomas Nikolaus has forthcoming work \cite{thomas-func} 
completely establishing the functoriality we postulate here.

Finally, if we had chosen to work with algebras instead of
categories (as is all we need in practice), 
then one could make do with the methods
of \cite{nikolaus-scholze}. But we have not used this approach
here because we find it to be not as well-suited 
as the categorical approach for the problems at hand.

\subsection{Review of traces}\label{ss:trace-review}

Let $\sC$ be dualizable in $\StCat_{cont}$.
Then the trace functor:

\[
\tr_{\sC}:\TwoEnd_{\StCat_{cont}}(\sC) \to \Sp
\]

\noindent satisfies:

\[
\tr_{\sC}(TS) = \tr_{\sC}(ST)
\]

\noindent for $S,T \in \TwoEnd_{\StCat_{cont}}(\sC)$, and
more generally:

\[
\tr_{\sC}(T_1\ldots T_n) = \tr_{\sC}(T_2 \ldots T_n T_1) 
\]

\noindent for $T_1,\ldots T_n \in \TwoEnd_{\StCat_{cont}}(\sC)$.
Here we are lazily writing an equals sign for an existence
of canonical isomorphism; and more functorially, we should
work with a series of functors indexed by a cyclic set. 

In particular, there is a $\bZ/n$-action on\footnote{For clarity:
throughout this section, e.g. $T^n$ denotes the $n$-fold composition
of $T$ with itself (and not, say, the $n$-fold product of it with itself).}
$\tr_{\sC}(T^n)$ for any $T$. As a variant,
we have a cyclic functor
with constant value $\tr_{\sC}(\id_{\sC})$ and whose underlying
simplicial functor is constant;
this recovers the usual $\bB \bZ$-action on $\THH$.

\begin{rem}\label{r:bb-bz}

Here we are using somewhat non-standard notation:
we use $\bB \bZ$ rather than $S^1$ or $\bT$ to emphasize that
the story is not at all transcendental. 
Note that in this perspective,
the usual homomorphism $\bZ/n \to S^1 = \bB \bZ$ corresponds to the
extension $0 \to \bZ \xar{n} \bZ \to \bZ/n$ of 
abelian groups.

\end{rem}

\subsection{}\label{ss:trace-review-2}

Next, recall the following additional functoriality of
traces. Note that we have used some of this
material already in the proof of Theorem \ref{t:k}.

Suppose $T \in \TwoEnd(\sC)$ and $S \in \TwoEnd(\sD)$.
Suppose moreover that 
$\psi:\sC \to \sD \in \StCat_{cont}$
is an exact functor between dualizable stable categories that
admits a continuous right adjoint, and that we are given
a natural transformation:

\[
\psi T \to S \psi.
\]

\noindent Then there is an induced map:

\[
\tr_{\sC}(T) \to \tr_{\sD}(S)
\]

\noindent satisfying expected compatibilities.

Note that in such a case, for $n\in \bZ^{>0}$ we also
obtain a natural transformation:

\[
\psi T^n \to S \psi T^{n-1} \to \ldots \to
S^{n-1} \psi T \to S^n \psi
\]

\noindent and so a map:

\[
\tr_{\sC}(T^n) \to \tr_{\sD}(S^n).
\]

\noindent By construction, this map is $\bZ/n$-equivariant.

\begin{example}\label{e:trace}

Suppose $\sC = \Sp$ and $T$ is the identity
functor. Then a functor $\psi$ as above is equivalent
to a compact object $\sF \in \sD$, and a natural transformation
as above is equivalent to a map $\eta:\sF \to S(\sF)$.
From this datum, the above constructs a canonical
map:

\[
\tr_{\Sp}(\id_{\Sp}) = \bS \to \tr_{\sD}(S)
\] 

\noindent i.e., it gives a point of $\Omega^{\infty} \tr_{\sD}(S)$.
For later use, we denote this point $\tr_{\sF}(\eta)$.

We remark that we have already seen this construction in
Example \ref{e:trace}.

\end{example}

\subsection{Calculation of $\THH$}

Suppose that $\sC \in \StCat_{cont}$ is dualizable
and equipped with a continuous endofunctor $T$.

\begin{prop}\label{p:thh}

There is a canonical $\bB \bZ$-equivariant isomorphism:

\[
\THH(\sC) \bigoplus
\oplus_{n> 0} \, \on{Ind}_{\bZ/n}^{\bB \bZ} \tr_{\sC}(T^n) \isom  
\THH(\SqZero(\sC,T)).
\]

\noindent Here $\on{Ind}_{\bZ/n}^{\bB \bZ}$ is the 
right adjoint induction functor from spectra with (naive) 
$\bZ/n$-actions
to spectra with $\bB \bZ$-actions.

\end{prop}

\begin{proof}

Here is the method. 

Suppose $\sD$ a dualizable stable category and we wish
to calculate its Hochschild homology. Note that
$\sD \otimes \sD^{\vee} \isom \TwoEnd(\sD)$, and the
trace map on the right hand side corresponds to the
canonical pairing on the left hand side.
Then we might try to calculate this tensor product
in some explicit terms, calculate what corresponds to the
identity functor for $\sD$, and then apply the evaluation
functor. This will be our approach for $\sD = \SqZero(\sC,T)$.

We remark that 
some of the manipulations below may also be
understood in terms of usual square-zero extensions (using 
Proposition \ref{p:sqzero-comparison}), and we encourage the reader to
do the exercise of translating.

\step 

First, we claim that for any $\sD \in \StCat_{cont}$,
the functor:

\begin{equation}\label{eq:sqzero-tens}
\SqZero(\sC,T) \otimes \sD \isom 
\SqZero(\sC \otimes \sD, T \otimes \id_{\sD})
\end{equation}

\noindent is an equivalence.
We will show both sides map to $\sC\otimes \sD$ monadically,
and the induced morphism of monads is an isomorphism.

Note that the functor 
$\SqZero(\sC,T) \xar{(\sF,\eta) \mapsto \Ker(\eta)} \sC$
admits a left adjoint equipping an object
of $\sC$ with the zero map to $T$ of itself. 
Moreover, the local nilpotence condition in the definition
of $\SqZero(\sC,T)$ implies that this functor is
conservative, and therefore (being continuous)
monadic. Note that the underlying monad on 
$\sC$ sends $\sF \in \sC$ to 
$\sF \oplus T(\sF)[-1] = \Ker(0:\sF \to T(\sF))$.

Then we recall that monadicity is preserved under
tensor products in $\StCat_{cont}$, so
the left hand side of \eqref{eq:sqzero-tens}
maps monadically to $\sC \otimes \sD$.
Moreover, applying the above to 
$\SqZero(\sC \otimes \sD, T \otimes \id_{\sD})$,
we obtain that this category maps monadically
to $\sC\otimes \sD$. Then it is immediate
to verify that the functor in \eqref{eq:sqzero-tens} 
intertwines these monadic functors and induces an
equivalence of monads on $\sC \otimes \sD$, and therefore
is an equivalence.

\step 

Next, we claim that 
$\SqZero(\sC,T)$ is dualizable
with dual $\SqZero(\sC^{\vee},T^{\vee})$. 
Here we recall that a functor $T: \sC \to \sC$
induces a dual functor $T^{\vee}: \sC^{\vee} \to \sC^{\vee}$;
explicitly, for 
$\lambda \in \sC^{\vee} = \TwoHom(\sC,\Sp)$, 
$T^{\vee}(\lambda) = \lambda \circ T$.

First, we construct the evaluation map:

\[
\SqZero(\sC,T) \otimes \SqZero(\sC^{\vee},T^{\vee}) \to \Sp.
\] 

\noindent It is equivalent to construct a functor:

\[
\SqZero(\sC,T) \times \SqZero(\sC^{\vee},T^{\vee}) \to \Sp
\]

\noindent commuting with colimits in each variable separately.
This pairing sends:

\[
\big((\sF,\eta:\sF \to T(\sF)),(\lambda,\mu:\lambda \to \lambda T)
\big)
\]

\noindent to:

\[
\on{Eq}\big(\lambda(\sF) 
\overset{\lambda(\eta)}{\underset{\mu}{\rightrightarrows}}
\lambda T(\sF) \big).
\]

Next, we define the coevaluation map:

\[
\Sp \to \SqZero(\sC,T) \otimes \SqZero(\sC^{\vee},T^{\vee}).
\]

\noindent For this, it is helpful to realize the
right hand side more explicitly using
the previous step.
Iteratively applying the previous step, we obtain:

\[
\begin{gathered}
\SqZero(\sC,T) \otimes \SqZero(\sC^{\vee},T^{\vee}) =
\SqZero(\sC \otimes \SqZero(\sC^{\vee},T^{\vee}), T \otimes \id) =
\\
\SqZero\big(\sC \otimes \sC^{\vee}, 
(T \otimes \id_{\sC}^{\vee}) \times 
(\id_{\sC} \otimes T^{\vee})\big).
\end{gathered}
\]

Noting that $\sC \otimes \sC^{\vee} \isom \TwoEnd(\sC)$
by duality, we obtain
that objects of the above tensor product
are the same as data:\footnote{
A posteriori, we have 
$\SqZero(\sC,T) \otimes \SqZero(\sC^{\vee},T^{\vee})  
\isom \TwoEnd(\SqZero(\sC,T))$.

Explicitly, the endofunctor of
$\SqZero(\sC,T)$ corresponding to the above
data sends $(\sF,\eta)$ to:

\[
\sF^{\prime} \coloneqq
\on{Eq}\big(
S(\sF) \overset{S(\eta)}{\underset{\alpha}{\rightrightarrows}} ST(\sF) 
\big) 
\]

\noindent with 
$\eta^{\prime}: \sF^{\prime} \to T(\sF^{\prime})$
induced by taking equalizers along rows in the (appropriately
commuting) diagram:

\[
\xymatrix{
S(\sF) \ar@<.4ex>[r]^{S(\eta)} \ar@<-.4ex>[r]_{\alpha} \ar[d]^{\beta}
& ST(\sF) \ar[d]^{\beta} \\
TS(\sF) \ar@<.4ex>[r]^{TS(\eta)} \ar@<-.4ex>[r]_{T(\alpha)} & TST(\sF).
}
\]
}

\begin{equation}\label{eq:s}
\vcenter{\xymatrix{
S \ar[r]^{\alpha} \ar[d]_{\beta} & ST \\
TS
}}
\end{equation} 

\noindent with:

\[
\begin{gathered}
\colim \, \big(S \xar{\alpha} ST \xar{ (-\circ T)(\alpha)} ST^2 \ldots\big) 
= 0\\
\colim \, \big( S \xar{\beta} TS \xar{ (T \circ -)(\beta)} T^2 S \ldots \big)
= 0.
\end{gathered}
\]

Now our coevaluation map is specified
by an object of the above tensor product (since it is
a continuous exact functor out of spectra).
In the graphical display of \eqref{eq:s}, this object is:

\[
\xymatrix{
\oplus_{n \geq 0} T^n \ar[r]^{\pi} \ar[d]_{\pi} & 
\oplus_{n \geq 1} T^n \\
\oplus_{n \geq 1} T^n
}
\]

\noindent where $\pi$ denotes the natural projection.

To verify that this actually defines a duality datum,
we should show that the composition:

\[
\SqZero(\sC,T) \xar{\id \otimes \on{coev}} 
\SqZero(\sC,T) \otimes \SqZero(\sC^{\vee},T^{\vee}) 
\otimes \SqZero(\sC,T) \xar{\on{ev} \otimes \id}
\SqZero(\sC,T)
\]

\noindent is isomorphic to the identity functor
(by the symmetry of $\sC$ and $\sC^{\vee}$ here,
this suffices).

This composition sends $(\sF,\eta) \in \SqZero(\sC,T)$ to:

\[
\sF^{\prime} = \on{Eq}
(\oplus_{n \geq 0} T^n(\sF) 
\overset{\oplus_n T^n(\eta)}{\underset{\pi}{\rightrightarrows}}
\oplus_{n \geq 1} T^n(\sF))
\]

\noindent equipped with the map 
$\eta^{\prime}:\sF^{\prime} \to T(\sF^{\prime})$
induced by the natural projection from $\sF^{\prime}$ to: 

\[
T(\sF^{\prime}) = 
\on{Eq}
(\oplus_{n \geq 1} T^n(\sF) \rightrightarrows \oplus_{n \geq 2} T^n(\sF)).
\]

\noindent We wish to construct a functorial isomorphism
$(\sF,\eta) \simeq (\sF^{\prime},\eta^{\prime})$.

First, we observe:\footnote{Note the difference
from the equalizer we aim to calculate: it is 
in indexing the second term in the equalizer.}

\[
\begin{gathered}
\on{Eq}
(\oplus_{n \geq 0} T^n(\sF) \underset{\id}{\rightrightarrows} \oplus_{n \geq 0} T^n(\sF)) =
\on{Coeq}
(\oplus_{n \geq 0} T^n(\sF) \rightrightarrows \oplus_{n \geq 0} T^n(\sF))[-1] = \\
\colim_{n \geq 0} T^n(\sF)[-1] = 0
\end{gathered}
\]

\noindent by local nilpotence of $\eta$. 
Therefore, the equalizer we are trying
to calculate is:

\[
\Coker\big(\on{Eq}(0 \rightrightarrows \sF) \to 
0\big) = \Coker(\sF[-1] \to 0) = \sF.
\]

\noindent The additional compatibility
between $\eta$ and $\eta^{\prime}$ is readily
seen.

\step

We now obtain a formula for Hochschild homology
as a bare spectrum by 
composing the evaulation and coevaluation maps.

First, note that the evaluation map
sends an object \eqref{eq:s} to:

\[
\on{Eq}(\tr_{\sC}(S) 
\overset{\alpha}{\underset{\beta}{\rightrightarrows}}
\tr_{\sC}(ST) \simeq \tr_{\sC}(TS)).
\]

\noindent Indeed, this follows by identifying the
two on ``pure tensors"
$(\sF,\eta)\boxtimes (\lambda,\mu)$ as in the
construction of the evaluation map.

We thus obtain:

\begin{equation}\label{eq:thh-fmla}
\THH(\SqZero(\sC,T)) \isom 
\underset{n \geq 0}{\oplus} 
\on{Eq}\big(\tr_{\sC}(T^n) 
\overset{\id}{\underset{\sigma_n}{\rightrightarrows}}
\tr_{\sC}(T^n)\big)
\end{equation}

\noindent where 
$\sigma_n:\on{\tr}_{\sC}(T^n) \isom \on{\tr}_{\sC}(T^n)$
is the action of the generator of $\bZ/n$ on
this trace (c.f. \S \ref{ss:trace-review}).

We observe that the summand:

\[
\on{Eq}\big(\tr_{\sC}(T^n) 
\overset{\id}{\underset{\sigma_n}{\rightrightarrows}}
\tr_{\sC}(T^n)\big)
\] 

\noindent is isomorphic as a spectrum
to $\on{Ind}_{\bZ/n}^{\bB \bZ} \tr_{\sC}(T^n)$.

Therefore, the left and right hand sides of \eqref{eq:thh-fmla}
have natural $\bB\bZ$-actions. It remains to show that
isomorphism of \eqref{eq:thh-fmla} upgrades
to a $\bB \bZ$-equivariant one.

\step 

First, as a (slightly\footnote{If the right
hand side were a product instead of a sum,
what we explain here would be adequate. And in
fact, for our application we may assume 
$\tr_{\sC}(T^n) \in \Sp^{\leq -n}$, which forces
the direct sum and direct product to coincide.})
toy version of the problem, fix $n>0$. We claim that 
the composition:

\[
\THH(\SqZero(\sC,T)) \xar{\eqref{eq:thh-fmla}} 
\on{Eq}\big(\tr_{\sC}(T^n) 
\overset{\id}{\underset{\sigma_n}{\rightrightarrows}}
\tr_{\sC}(T^n)\big) \to 
\tr_{\sC}(T^n)
\]

\noindent is $\bZ/n$-equivariant,
and that the induced map:

\[
\THH(\SqZero(\sC,T)) \to \on{Ind}_{\bZ/n}^{\bB \bZ} \tr_{\sC}(T^n) =
\on{Eq}\big(\tr_{\sC}(T^n) 
\overset{\id}{\underset{\sigma_n}{\rightrightarrows}}
\tr_{\sC}(T^n)\big)
\]

\noindent is the natural projection arising from 
\eqref{eq:thh-fmla}.

Note that we have a natural functor:

\[
\begin{gathered}
\Oblv:\SqZero(\sC,T) \to \sC \\
(\sF,\eta) \mapsto \sF
\end{gathered}
\]

\noindent that admits a continuous right 
adjoint.\footnote{Explicitly,
this right adjoint sends $\sG \in \sC$ to 
$\oplus_{n \geq 0} T^n(\sG)$ equipped with 
the projection map $\oplus_{n \geq 0} T^n(\sG) \to 
T(\oplus_{n \geq 0} T^n(\sG)) = \oplus_{n \geq 1} T^n(\sG)$.}
Moreover, there is a canonical natural transformation:

\begin{equation}\label{eq:eta-nat-trans}
\Oblv \to T \circ \Oblv 
\end{equation}

\noindent that evaluates on $(\sF,\eta)$ as the map $\eta$.

By \S \ref{ss:trace-review-2}, we obtain a natural 
$\bZ/n$-equivariant map:

\[
\THH(\SqZero(\sC,T)) = 
\tr_{\SqZero(\sC,T)}(\id_{\SqZero(\sC,T)}) \to 
\tr_{\sC}(T^n).
\]

\noindent It is straightforward to see that this map has the desired
compatibilities with \eqref{eq:thh-fmla}.

\step 

Finally, we explain how to complete the argument. We do this using
a general format for $\THH$ (and traces more generally) to have
gradings.\footnote{The present discussion admits some natural extensions,
which we highlight here.

First, there is a notion of \emph{graded cyclotomic spectrum}, which
is discussed in \S \ref{ss:cyc-filt}. (The key point is
that the Frobenius at $p$ multiplies degrees by $p$.)
In particular, up to adapting \cite{amr} to the graded setting,
the present discussion shows that $\THH$ of a graded object
of $\StCat_{cont}$ is a graded cyclotomic spectrum.

Moreover, similar ideas may developed in the filtered setting.}

Let $\Rep(\bG_m)$ denote the symmetric monoidal category of $\bZ$-graded spectra
with the convolution\footnote{I.e., if $\sF$ and $\sG$ are spectra,
then $\sF(n) \otimes \sG(m) = (\sF\otimes \sG)(n+m)$, where
e.g. $\sF(n)$ indicates we consider $\sF$ as graded purely in degree $n$,
on the left hand side $\otimes$ indicates our convolution monoidal structure,
and on the right hand side it indicates the usual tensor product.}
 monoidal structure.\footnote{This category
is readily seen to in fact be comodules over 
the (bi-$\sE_{\infty}$) Hopf algebra $\Sigma^{\infty}\bZ$.}
A \emph{grading} on $\sD \in \StCat_{cont}$ is the datum of 
a category $\sD^{gr} \in \StCat_{cont}$ (of ``graded objects in $\sD$") equipped with 
$\Rep(\bG_m)$-module category structure and an isomorphism:

\[
\sD^{gr} \underset{\Rep(\bG_m)}{\otimes} \Sp \isom \sD.
\]

\noindent Here we are using the symmetric monoidal functor
$\Rep(\bG_m) \to \Sp$ of forgetting the 
grading.\footnote{Here is another language for categorical gradings, which
the reader may safely skip. We use some terminology and notation that we do not
wish to explain here.
 
Note that \cite{shvcat} Theorem 2.2.2
is true for spectra in the special case $G = \bG_m$: the proof from
\emph{loc. cit}. \S 7.2 works in this setup. (This is closely related to
$\bG_m$ being linearly reductive in any characteristic.)

The upshot is that a 
grading on $\sD$ is equivalent to a weak $\bG_m$-action on $\sD$,
where one recovers $\sD^{gr}$ as $\sD^{\bG_m,w}$. 

For example, in slightly imprecise terms, 
the relevant weak $\bG_m$-action on $\SqZero(\sC,T)$ that
we use below scales the map $\eta$.}

We claim that if $\sD$ is dualizable and equipped with 
a grading, then $\THH(\sD)$ is naturally graded
spectrum. 

Indeed, one can show\footnote{E.g., using the previous footnote and
standard techniques.}
that $\sD^{gr}$ is automatically dualizable as a $\Rep(\bG_m)$-module category.
Therefore, we can take its $\THH$ (the trace of the identity) 
in this category to obtain an object:

\[
\THH_{/\Rep(\bG_m)}(\sD^{gr}) \in \Rep(\bG_m).
\]

\noindent Then by functoriality, this object maps to 
$\THH(\sD)$ under the forgetful functor $\Rep(\bG_m) \to \Sp$, i.e., 
it induces a grading on $\THH(\sD)$.
Moreover, this construction makes manifest functoriality of traces
in the graded setting, similar to \S \ref{ss:trace-review}; we do
not spell out the details here.

We apply this to $\sD = \SqZero(\sC,T)$. 
We set $\SqZero(\sC,T)^{gr}$ to be the category
whose objects are collections $\sF_n \in \sC$ for each $n \in \bZ$
and equipped with maps $\eta_n:\sF_n \to T(\sF_{n+1})$ that are locally
nilpotent in the sense that $\colim_n T^n(\sF_{n+m}) = 0$ for any $m$.

This category has an obvious action of $\Rep(\bG_m)$, and the functor:

\[
\begin{gathered}
\SqZero(\sC,T)^{gr} \to \SqZero(\sC) \\
\big((\sF_n)_{n \in \bZ},(\eta_n)_{n\in \bZ}\big) \mapsto 
(\oplus_n \sF_n,\oplus_n \eta_n)
\end{gathered}
\]

\noindent naturally upgrades to an equivalence:

\[
\SqZero(\sC,T)^{gr} \underset{\Rep(\bG_m)}{\otimes} \Sp \isom  
\SqZero(\sC,T).
\]

Then it is straightforward to verify that the grading on 
$\THH(\SqZero(\sC,T))$ coming from the above coincides with
the grading appearing in \eqref{eq:thh-fmla}. 

Now note that
the natural transformation \eqref{eq:eta-nat-trans} is
\emph{graded of degree 1} in the natural sense. 
Therefore, we can apply the method from the previous step
to see that the degree $n$ part of $\THH(\SqZero(\sC,T))$
is $\bB \bZ$-equivariantly isomorphic to 
$\on{Ind}_{\bZ/n}^{\bB \bZ} \tr_{\sC}(T^n)$, completing
the argument. 

\end{proof}

\subsection{Cyclotomic structure}

In Proposition \ref{p:thh}, for $\sC$ dualizable and
$T:\sC \to \sC$ continuous and exact, we calculated 
$\THH(\SqZero(\sC,T))$ with its $\bB \bZ$-action.
We now wish to describe its cyclotomic structure.

\subsection{}\label{ss:tate-diag}

First, we need some additional functoriality for traces.
Fix $p$ a prime. Then we claim that there is a \emph{Tate diagonal} map:

\[
\Delta_p: \tr_{\sC}(T) \to \tr_{\sC}(T^p)^{t\bZ/p}
\]

\noindent functorial in $T$ (actually, satisfying functoriality
as in \S \ref{ss:trace-review-2}, but we do not need this).

First, note that 
$\TwoEnd(\sC) \xar{T \mapsto \tr_{\sC}(T^p)^{t\bZ/p}} 
\Sp$ 
is exact. The argument\footnote{We
remark that this argument is a variant
of the standard combinatorial proof of Fermat's little theorem;
see Wikipedia for example.} is standard, c.f. \cite{nikolaus-scholze} Proposition III.1.1. 

By Lemma \ref{l:nat-trans-sp}, it suffices to construct
the natural transformation of functors to $\Gpd$
obtained by applying $\Omega^{\infty}$.
Moreover, as $T \mapsto \tr_{\sC}(T)$ commutes with all colimits,
it suffices to define the restriction our natural transformation
when restricted along:

\[
\sC \times \sC^{\vee} \to \sC \otimes \sC^{\vee} = \TwoEnd(\sC).
\]

Now for $(\sF,\lambda) \in \sC \times \sC^{\vee}$,
the trace of the corresponding functor is
$\lambda(\sF)$, while the trace of its $p$-fold
composition is
$\lambda(\sF)^{\otimes p}$. Then we use the natural map:

\[
\begin{gathered}
\Omega^{\infty}\lambda(\sF) \isom  
\Big(\big(\Omega^{\infty}\lambda(\sF)\big)^p\Big)^{h\bZ/p} \to
\big(\Omega^{\infty}\lambda(\sF)^{\otimes p}\big)^{h\bZ/p} = \\
\Omega^{\infty}\lambda(\sF)^{\otimes p,h\bZ/p} \to 
\Omega^{\infty}\lambda(\sF)^{\otimes p,t\bZ/p}. 
\end{gathered}
\]

\begin{example}

For $\sC = \Sp$, a functor $T$ of this type is necessarily the
tensor product with some spectrum. In this case, the
Tate diagonal construction above recovers that
of \cite{nikolaus-scholze} \S III.1.

\end{example}

\begin{variant}\label{v:diag-equivariant}

Generalizing \cite{nikolaus-scholze} \S III.3,
there is some additional functoriality.
For example, for $n>0$, the Tate diagonal map:

\[
\Delta_p:\tr_{\sC}(T^n) \to \tr_{\sC}(T^{np})^{t\bZ/p}
\]

\noindent is naturally $\bZ/n$-equivariant,
where we use
the natural $(\bZ/np)/(\bZ/p) = \bZ/n$-action on the right hand
side.

\end{variant}

\subsection{}\label{ss:cyc-str}

We now construct a cyclotomic structure on the
$\bB\bZ$-spectrum:

\[
\underset{n \geq 1}{\oplus} \, \on{Ind}_{\bZ/n}^{\bB \bZ} \tr_{\sC}(T^n).
\]

\noindent We remark that a version of this construction appears
in \cite{lindenstrauss-mccarthy} in a related context
(but using the formalism of equivariant homotopy theory).

For simplicity, we assume $\tr_{\sC}(T^n) \in \Sp^{\leq 0}$
for all $n$ so we are in the setting of
\cite{nikolaus-scholze} (the general non-connective setting can be
treated following \cite{amr}). So for every prime $p$,
we need to construct a suitable ``Frobenius" map. 

For any integer $n$, we have the $\bZ/n$-equivariant Tate diagonal:

\[
\tr_{\sC}(T^n) \to \tr_{\sC}(T^{np})^{t\bZ/p}. 
\]

\noindent We then induce to $\bB\bZ$-representations:

\[
\on{Ind}_{\bZ/n}^{\bB \bZ}\tr_{\sC}(T^n) \to 
\on{Ind}_{\bZ/n}^{\bB \bZ}\big(\tr_{\sC}(T^{np})^{t\bZ/p}\big)
\]

\noindent and observe\footnote{To see this, suppose in 
generality that we are given:

\[
\xymatrix{
K \ar@{=}[d] \ar[r] & H_1 \ar[d] \ar[r] & G_1 \ar[d] \\
K \ar[r] & H_2 \ar[r] & G_2
}
\]

\noindent where the rows are fiber sequences of groups, 
the maps
$H_i \to G_i$ are epimorphisms (i.e., surjective on $\pi_0$),
and $K$ is 
a finite (discrete) group.

Then for $V$ a spectrum with a (naive) $H_1$-action, there
is an obvious commuting diagram of spectra with $G_2$-actions:

\[
\xymatrix{
\Ind_{H_1}^{H_2}(V)_{hK} \ar[r] \ar[d] & \Ind_{G_1}^{G_2}(V_{hK}) \ar[d] \\
\Ind_{H_1}^{H_2}(V)^{hK} & \Ind_{H_1}^{H_2}(V^{hK}) \ar[l]
}
\]

\noindent where the vertical maps come from norm maps for $K$.

We take $K = \bZ/p \to H_1 = \bZ/np \to G_1 = \bZ/p$ and
$K = \bZ/p \to H_2 = \bB \bZ \xar{p} G_2 = \bB \bZ$ as our rows,
with the natural maps relating them. Then the 
morphism in the top row of our diagram above is an isomorphism,
since the functor $\Ind_{\bZ/m}^{\bB \bZ}$ commutes with colimits
for any $m$. Passing to kernels along the vertical arrows then
gives the desired map.}
that there is a natural map: 

\[
\on{Ind}_{\bZ/n}^{\bB \bZ}\big(\tr_{\sC}(T^{np})^{t\bZ/p}\big) \to
\big(\on{Ind}_{\bZ/np}^{\bB \bZ}\tr_{\sC}(T^{np})\big)^{t\bZ/p}
\]

\noindent that is equivariant for the multiplication 
by $p$-map $\bB\bZ \to \bB \bZ$.

Composing the above morphisms and taking 
the direct sum over $n$, we obtain:

\[
\vph_p:
\underset{n \geq 1}{\oplus} \on{Ind}_{\bZ/n}^{\bB \bZ}\tr_{\sC}(T^n) \to 
\underset{n \geq 1}{\oplus} 
\big(\on{Ind}_{\bZ/np}^{\bB \bZ}\tr_{\sC}(T^{np})\big)^{t\bZ/p} 
\]

\noindent which is again equivariant against $p:\bB \bZ \to \bB \bZ$.
By \cite{nikolaus-scholze}, 
these maps over all $p$ define a cyclotomic structure assuming connectivity.
(And again, refining this construction somewhat gives a cyclotomic
structure in general, following \cite{amr}.)

Tracing the constructions, we have:

\begin{lem}\label{l:cyclotomic}

This cyclotomic structure is
the canonical one on:

\[
\THH_{red}(\SqZero(\sC,T)) \coloneqq 
\Ker\big(\THH(\SqZero(\sC,T)) \to \THH(\sC)\big)
\]

\noindent under the isomorphism of Proposition \ref{p:thh}.

\end{lem}

\subsection{Genuine fixed points}\label{ss:genuine}

It is convenient to introduce the following notation.

For $p$ a prime, let $\tr_{\sC}(T^p)^{\bZ/p} \in \Sp$ denote the
\emph{genuine fixed points}, which by definition is
the fiber product:

\[
\xymatrix{
\tr_{\sC}(T^p)^{\bZ/p} \ar[r] \ar[d] & \tr_{\sC}(T^p)^{h\bZ/p} \ar[d] \\
\tr_{\sC}(T) \ar[r]^{\Delta_p} & \tr_{\sC}(T^p)^{t\bZ/p}.
}
\]

More generally, for $n \geq 1$, we construct the
genuine fixed points $\tr_{\sC}(T^n)^{\bZ/n}$
as the iterated fiber product:

\[
\begin{gathered} 
\tr_{\sC}(T^n)^{\bZ/n} \coloneqq \\
\tr_{\sC}(T) 
\underset{
\underset{p \text{ prime}}{\underset{p\mid n}{\prod}}
\tr_{\sC}(T^p)^{t\bZ/p}
}{\bigtimes} 
\underset{p \text{ prime}}{\underset{p\mid n}{\prod}}
\tr_{\sC}(T^p)^{h\bZ/p} 
\ldots \underset{
\underset{p \text{ prime}}
{\underset{\varpi(d) = k}{{\underset{pd\mid n}{\prod}}}}
\tr_{\sC}(T^{pd})^{t\bZ/p,h\bZ/d}
}
{\bigtimes}
\underset{\varpi(d) = k+1}{\underset{d\mid n}{\prod}}
\tr_{\sC}(T^d)^{h\bZ/d} \ldots \\
\ldots 
\underset{
\underset{p \text{ prime}}
{\underset{\varpi(d) = \varpi(n)-1}{{\underset{pd\mid n}{\prod}}}}
\tr_{\sC}(T^{pd})^{t\bZ/p,h\bZ/d}
}
{\bigtimes}
\tr_{\sC}(T^n)^{h\bZ/n}.
\end{gathered}
\]

\noindent Here for $d \in \bZ^{>0}$, we let 
$\varpi(d) = \sum_{p \text{ prime}} v_p(d)$. The structure
maps in the above iterated fiber product are constructed
as follows. Going right we use the Tate diagonal maps:

\[
\tr_{\sC}(T^d)^{h\bZ/d} \to \tr_{\sC}(T^{pd})^{t\bZ/p,h\bZ/d} 
\]

\noindent coming from Variant \ref{v:diag-equivariant}.
And going left we simply use the canonical projection from invariants
to the Tate construction.

Note that for any $d\mid n$, there is a canonical restriction map:

\[
\tr_{\sC}(T^n)^{\bZ/n} \to \tr_{\sC}(T^d)^{\bZ/d}. 
\]

\begin{lem}\label{l:norm-cofib}

For any $n$ there is a canonical isomorphism:

\[
\xymatrix{
\tr_{\sC}(T^n)_{h\bZ/n} \isom 
\Ker\big(\tr_{\sC}(T^n)^{\bZ/n} \to 
\underset{d\mid n}{\lim} \, \tr_{\sC}(T^d)^{\bZ/d}\big).
}
\]

\end{lem}

\begin{proof}

The limit on the right is calculated as an iterated fiber
product as in the definition of $\tr_{\sC}(T^n)$, but where we omit
the last fiber product in its definition.
Therefore, the kernel we are trying to compute coincides with:

\[
\Ker\Big(\tr_{\sC}(T^n)^{h\bZ/n} \to 
\underset{p \text{ prime}}
{\underset{\varpi(d) = \varpi(n)-1}{{\underset{pd\mid n}{\prod}}}}
\tr_{\sC}(T^{pd})^{t\bZ/p,h\bZ/d}\Big)
\]

\noindent Note that on the right, $pd$ is necessarily equal to $n$,
so we can rewrite this expression as:

\[
\Ker\Big(\tr_{\sC}(T^n)^{h\bZ/n} \to 
\underset{p \text{ prime}}
{\underset{p\mid n}{\prod}}
\tr_{\sC}(T^n)^{t\bZ/p,h\bZ/(n/p)}\Big).
\]

Now we observe that for any $V \in \Sp^{\leq 0}$
with a (naive) $\bZ/n$-action, the map:

\[
V^{t\bZ/n} \to 
\underset{p \text{ prime}}
{\underset{p\mid n}{\prod}}
V^{t\bZ/p,h\bZ/(n/p)}
\]

\noindent is an isomorphism. Indeed, it is easy to 
see\footnote{E.g., one notes that $V^{t\bZ/n}$ is $n$-adically complete and
shows that $V^{t\bZ/p^{v_p(n)},h\bZ/(n/p^{v_p(n)})}$ is
its $p$-adic completion.}
that:

\[
V^{t\bZ/n} \isom \prod_{p \text{ prime}} 
V^{t\bZ/p^{v_p(n)},h\bZ/(n/p^{v_p(n)})}
\]

\noindent for arbitrary $V$, and in the connective case we can further
apply \cite{nikolaus-scholze} Lemma II.4.1, which is a version of
the Tate orbit lemma.

Therefore, by our connectivity assumption on traces
of powers of $T$, we need to calculate:

\[
\Ker\Big(\tr_{\sC}(T^n)^{h\bZ/n} \to 
\tr_{\sC}(T^n)^{t\bZ/n}
\Big)
\]

\noindent which is certainly $\tr_{\sC}(T^n)_{h\bZ/n}$.

\end{proof}

\subsection{Calculation of $\TC$}

We use the above as follows.

\begin{thm}\label{t:tc-calc}

Suppose that for every $n > 0$, $\tr_{\sC}(T^n) \in \Sp^{\leq -n}$.

Then there is a natural isomorphism:

\[
\TC_{red}(\SqZero(\sC,T)) \coloneqq 
\Ker\big(\TC(\SqZero(\sC,T)) \to \TC(\sC)\big) \isom 
\underset{n}{\lim} \, \tr_{\sC}(T^n)^{\bZ/n}
\]

\noindent where the limit is over positive integers
ordered under divisibility.

In particular, $\TC_{red}(\SqZero(\sC,T))$ has a 
complete decreasing filtration indexed by positive integers
under divisibility, and there is a canonical isomorphism:

\[
\gr_n \TC_{red}(\SqZero(\sC,T)) \isom 
\tr_{\sC}(T^n)_{h\bZ/n}.
\]

\end{thm}

\begin{rem}

This result is implicit in \cite{lindenstrauss-mccarthy}.   

\end{rem}

\begin{proof}[Proof of  Theorem \ref{t:tc-calc}]

By the Nikolaus-Scholze formula \cite{nikolaus-scholze} Corollary 1.5
for $\TC$ (using the connectivity assumption of \S \ref{ss:cyc-str}), 
we have:

\[
\TC_{red}(\SqZero(\sC,T)) = 
\on{Eq}\big(\THH_{red}(\SqZero(\sC,T))^{h\bB \bZ} \rightrightarrows
\prod_{p \text{ prime}}\THH_{red}(\SqZero(\sC,T))^{t\bZ/p,h\bB \bZ} \big).
\]

\noindent Here we recall that on the right hand side, 
we are taking the residual $\bB \bZ$-action using the
natural isomorphism 
$\bB \bZ = \Coker(\bZ/p \to \bB \bZ)$.\footnote{Here 
we are taking the cokernel in the category of
$\sE_{\infty}$-groups. In a less commutative setting, 
it would be better to note that 
there is a fiber sequence $\bZ/p \to \bB \bZ \xar{p} \bB \bZ$
with the right map surjective on $\pi_0$; this is the appropriate
notion of ``group quotient" in the homotopical setting.} 
We further recall that one map in the equalizer uses the
cyclotomic Frobenius, and the other uses the tautological
projection from homotopy $\bZ/p$-invariants to the Tate construction.

Now under our assumption of increasing connectivity on 
the traces of powers of $T$,
we have:

\begin{equation}\label{eq:sum=prod}
\THH_{red}(\SqZero(\sC,T) \xar{\overset{\text{Prop. \ref{p:thh}}}{\simeq}}
\underset{n}{\oplus} \, \on{Ind}_{\bZ/n}^{\bB \bZ} \tr_{\sC}(T^n) = 
\underset{n}{\prod} \, \on{Ind}_{\bZ/n}^{\bB \bZ} \tr_{\sC}(T^n).
\end{equation}

\noindent Therefore, we have:

\[
\begin{gathered}
\THH_{red}(\SqZero(\sC,T))^{h\bB \bZ} = 
\big(\underset{n}{\prod} \, \on{Ind}_{\bZ/n}^{\bB \bZ} \tr_{\sC}(T^n)\big)^{h \bB \bZ} = 
\underset{n}{\prod} \on{Ind}_{\bZ/n}^{\bB \bZ} \tr_{\sC}(T^n)^{h \bB \bZ} = \\
\underset{n}{\prod} \, \on{Ind}_{\bZ/n}^{\bB \bZ} \tr_{\sC}(T^n)^{h \bB \bZ} = 
\underset{n}{\prod} \, \tr_{\sC}(T^n)^{h \bZ/n}.
\end{gathered}
\]

For $p$ prime, we obtain:

\[
\THH_{red}(\SqZero(\sC,T))^{t\bZ/p,h\bB \bZ} =
\underset{n}{\prod} \, 
\big(\on{Ind}_{\bZ/n}^{\bB \bZ} \tr_{\sC}(T^n)\big)^{t\bZ/p,h \bB \bZ}.
\]

\noindent First, note that the factor $n$ vanishes if 
$p$ does not divide $n$. Indeed, in this case we have:

\[
\big(\on{Ind}_{\bZ/n}^{\bB \bZ} \tr_{\sC}(T^n)\big)^{t\bZ/p,h \bB \bZ} =
\big(\on{Ind}_{\bZ/pn}^{\bB \bZ} \on{Ind}_{\bZ/n}^{\bZ/pn} 
\tr_{\sC}(T^n)\big)^{t\bZ/p,h \bB \bZ} = 
\Big(\on{Ind}_{\bZ/pn}^{\bB \bZ} \big(\on{Ind}_{\bZ/n}^{\bZ/pn} 
\tr_{\sC}(T^n)\big)^{t\bZ/p}\Big)^{h \bB \bZ} = 0.
\]

\noindent And if $p$ does divide $n$, then
we have:

\[
\on{Ind}_{\bZ/n}^{\bB \bZ} \tr_{\sC}(T^n)\big)^{t\bZ/p,h \bB \bZ} =
\tr_{\sC}(T^n)^{t\bZ/p,h \bZ/(n/p)}.
\]

By Lemma \ref{l:cyclotomic}, the Frobenius map at $p$ is given
by the product over $n$ of the maps:

\[
\tr_{\sC}(T^n)^{h \bZ/n} \to \tr_{\sC}(T^{pn})^{t\bZ/p,h \bZ/n}.
\]

Now the fact that the equalizer above is the 
limit of genuine fixed points is formal from the definition
of genuine fixed points.
Moreover, the associated graded term was already calculated
in Lemma \ref{l:norm-cofib}.

\end{proof}

\begin{rem}

The mild connectivity assumption that $\tr_{\sC}(T^n) \in \Sp^{\leq 0}$
for all $n$ played an inessential role in our calculation of the
cyclotomic structure on $\THH(\SqZero(\sC,T))$
and in the definition of genuine fixed
points $\tr_{\sC}(T^n)^{\bZ/n}$: it is straightforward to generalize
to the non-connective setting here. In other words, we were
merely lazy there (the cost being our use of the Tate orbit lemma).

However, in the proof of the above theorem, the harsher
assumption that 
the connectivity of $\tr_{\sC}(T^n)$ tends to $\infty$ with $n$
played an essential role above. Indeed, it was crucially used 
in \eqref{eq:sum=prod}, which for example allowed us to
compute the homotopy $\bB\bZ$-invariants termwise. 

\end{rem}

\subsection{Proof of the main theorem}

We can now prove the main result of this section.

\begin{proof}[Proof of Theorem \ref{t:tc}]

For (essentially notational) convenience, we begin by ignoring
the compatibility with $K$-theory but treating the
general categorical setup, using Proposition \ref{p:sqzero-comparison}.

As $t$-structures are used in Theorem \ref{t:tc}, we define
a $t$-structure on $\TwoEnd(\sC)$ by setting
$T \in \TwoEnd(\sC)^{\leq 0}$ if $\tr_{\sC}(T^n) \in \Sp^{\leq -n}$
for every $n>0$. Note that for $A$ a connective
$\sE_1$-algebra and $M \in A\bimod^{\leq 0}$, 
we have:\footnote{Here we are using the $t$-structure
we just constructed. So in other words, there are 
more connective objects in this $t$-structure than
the usual one on $\TwoEnd(A\mod) = A\bimod$.}

\[
T_M \in \TwoEnd(A\mod)^{\leq 0}
\]

\noindent as $T_M \coloneqq M\otimes_A - [1]$ gives
so $\tr_{A\mod}(T_M^n) = \THH(A,M^{\otimes n}[n])$.

Now first observe that the functor:

\[
\begin{gathered}
\TwoEnd(\sC) \to \Sp \\
T \mapsto \tr_{\sC}(T^n)_{h\bZ/n}
\end{gathered}
\]

\noindent commutes with sifted colimits. Indeed, $T \mapsto T^n$
commutes with sifted colimits (as composition of functors commutes
with colimits in each variable), and traces and coinvariants
both commute with all colimits. 

Moreover, for $n \neq 1$, we claim that
the Goodwillie derivative of this
functor vanishes. Again, by exactness of traces and 
coinvariants, it suffices to show this for the functor sending $T$ to
its $n$-fold self-composition $T^n$. 
Here it is straightforward\footnote{See \cite{higheralgebra}
Proposition 6.1.3.4 for a statement in the general setting of
Goodwillie calculus.}
to see that the natural map:

\[
T^n \to \Omega((\Sigma T)^n) = \Omega(\Sigma^n(T^n)) = \Sigma^{n-1} (T^n)
\]

\noindent is nullhomotopic (and naturally so in $T$).

Therefore, by Lemma \ref{l:norm-cofib} and induction,
the functor:\footnote{Here we ask the reader to believe that
Lemma \ref{l:norm-cofib} is true in the non-connective setting
for appropriate definition of genuine fixed points, or to
read $\TwoEnd(\sC)$ as $\TwoEnd(\sC)^{\leq 1}$.}

\[
\begin{gathered}
\TwoEnd(\sC) \to \Sp \\
T \mapsto \tr_{\sC}(T^n)^{\bZ/n} 
\end{gathered}
\]

\noindent commutes with sifted colimits, the canonical projection
$\tr_{\sC}(T^n)^{\bZ/n} \to \tr_{\sC}(T)$ realizes
the right hand side as the Goodwillie derivative of the
left hand side (as functors of $T$).

Next, observe that for $T \in \TwoEnd(\sC)^{\leq 0}$, we have:

\[
\Ker\big(\TC_{red}(\SqZero(\sC,T)) \to \tr_{\sC}(T^{n!})^{\bZ/n!}\big)
\in \Sp^{<-n}
\]

\noindent by Theorem \ref{t:tc-calc}, as 
$\gr_m \TC_{red}(\SqZero(\sC,T)) = 
\tr_{\sC}(T^m)_{h\bZ/m} \in \Sp^{< -n}$ for $m > n$.
From here the results are formal:
by left completeness of the $t$-structure on $\Sp$,
we obtain Theorem \ref{t:tc} \eqref{i:tc-conn},
and similarly the identification of the Goodwillie derivative
from Theorem \ref{t:tc} \eqref{i:tc-deriv}.
Moreover, Lemma \ref{l:ext'n-left-complete} 
and the observation that each 
$\TwoEnd(\sC)^{\leq 0} \xar{T \mapsto \tr_{\sC}(T^n)^{\bZ/n}}
\Sp $ is manifestly extensible imply
pseudo-extensibility of $\TC_{red}$, i.e., 
Theorem \ref{t:tc} \eqref{i:ps-ext}.
 
It now remains to show the compatibility with the cyclotomic
trace from $K$-theory and the 
identifications of Goodwillie derivatives.
For this, we suppose that $\sC$ is compactly generated.
Recall from \S \ref{ss:ct} that in this case we have 
$\SqZero(\sC,T) \subset \sC^T$, with this inclusion preserving
compact objects. We will show that the diagram:

\[
\xymatrix{
K(\SqZero(\sC,T)^c) \ar[r] \ar[d] & K(\sC^{T,c}) \ar[d] \\
\TC(\SqZero(\sC,T)) \ar[r] & \tr_{\sC}(T)
}
\]

\noindent commutes, where the left vertical arrow is the
cyclotomic trace map, and the bottom and right arrows
are maps to Goodwillie derivatives (using the above
for $\TC$ and Theorem \ref{t:k-deriv} for $K$-theory).
This suffices by Lemma \ref{l:ct=sqzero}.

Now observe that (by construction) the map
$\TC(\SqZero(\sC,T)) \to \tr_{\sC}(T)$ factors
through $\THH(\SqZero(\sC,T)$.
Hence, we can show the commutation of the above diagram 
with $\THH$ in place of $\TC$, and the Dennis trace
replacing the cyclotomic trace.

By Lemma \ref{l:nat-trans-sp},
it suffices to show the commutation of the
above diagram functorially in $T$ after applying $\Omega^{\infty}$.
Moreover, by Lemma \ref{l:dk-init}, it suffices to show the commutation
of the diagram:

\[
\xymatrix{
\SqZero(\sC,T)^{c,\simeq} \ar[r] \ar[d] & \Omega^{\infty} K(\sC^{T,c}) \ar[d] \\
\Omega^{\infty} \THH(\SqZero(\sC,T)) \ar[r] & \Omega^{\infty} \tr_{\sC}(T).
}
\]

Now suppose $(\sF,\eta) \in \SqZero(\sC,T)^c$, i.e.,
$\sF \in \sC^c$ and $\eta:\sF \to T(\sF)$ is locally nilpotent.
Recall that Example \ref{e:trace} produced a point
$\tr_{\sF}(\eta) \in \Omega^{\infty} \tr_{\sC}(T)$,
i.e., \emph{loc. cit}. gave a map
$\SqZero(\sC,T)^{c,\simeq} \to  \Omega^{\infty} \tr_{\sC}(T)$. We
claim that each leg of the above diagram identifies with
this map. 

For the upper leg of the diagram, this is tautological from
the proof of Lemma \ref{l:dk-init}.

We now treat the lower leg. The Dennis trace map applied
to $(\sF,\eta)$ produces a point
of $\Omega^{\infty} \THH(\SqZero(\sC,T))$, which (by definition of
the Dennis trace) is $\tr_{(\sF,\eta)}(\id_{\sF,\eta})$ in the notation
of Example \ref{e:trace}.
The proof of Proposition \ref{p:thh}
shows that its image under $\Omega^{\infty}$ of the map:\footnote{Note
that we use a direct sum indexed by $n \geq 0$ in the last term;
the $0$-fold composition of $T$ with itself is the identity functor,
so the leading term here is $\THH(\sC)$.}

\[
\THH(\SqZero(\sC,T)) = 
\THH(\sC) \bigoplus
\oplus_{n> 0} \, \on{Ind}_{\bZ/n}^{\bB \bZ} \tr_{\sC}(T^n)
\to 
\oplus_{n\geq 0} \, \tr_{\sC}(T^n) 
\]

\noindent is 
$(\tr_{\sF}(\id_{\sF}),\tr_{\sF}(\eta),\tr_{\sF}(T(\eta)\eta),\ldots)$
(which genuinely gives a point of the direct sum by local nilpotence).
Clearly the projection of this point to $\tr_{\sC}(T)$ is
$\tr_{\sF}(\eta)$.

\end{proof}

\section{Reduction to the split square-zero case}\label{s:gen'l}

\subsection{}

In this section, we prove the general form of Theorem \ref{t:dgm}.

\subsection{}

We return to the general format of Theorem \ref{t:conn-algs}
and Corollary \ref{c:full-vanishing},
letting $\Psi:\Alg_{conn} \to \Sp$ be a functor.
Our present goal is to axiomatize when $\Psi$ is constant
along general nilpotent extensions.

\subsection{Square-zero extensions}

We briefly review the theory of 
square-zero in the homotopical setting. 

Let $A \in \Alg$ be fixed, and let $m:A \otimes A \to A$
denote the multiplication. For $I \in A\bimod$,
a \emph{square-zero extension} of $A$ by $I$ is
by definition a morphism $\delta:\Ker(m) \to I[1] \in A\bimod$.

To obtain an algebra from such a datum, recall that
$\delta$ is equivalent by adjunction
to a morphism $A \to \SqZero(A,I[1]) \in \Alg$. We then set the underlying
algebra of $(A,I,\delta)$ to be the fiber product:

\[
A \underset{\SqZero(A,I[1])}{\times} A
\]

\noindent where one of the structure morphisms is induced
by $\delta$ and the other by the zero map $0 \to I[1]$
(noting $A = \SqZero(A,0)$). For example, for
$\delta = 0$ we recover the split square-zero extension
$\SqZero(A,I)$. 

This construction enhances to define a category 
$\Alg^{\SqZero}$ whose objects are data
$(A,I,\delta)$, and where morphisms
$(A_1,I_1,\delta_1) \to (A_2,I_2,\delta_2)$
consist of a morphism $f:A_1 \to A_2$ of algebras
and a commutative diagram of $A_1$-bimodules:

\[
\xymatrix{
\Ker(m_1) \ar[d]^{\delta_1} \ar[r]^f & \Ker(m_2) \ar[d]^{\delta_2} \\
I_1[1] \ar[r] & I_2[1].
}
\]

We let $\Alg_{conn}^{\SqZero}$ be the full subcategory
of $\Alg^{\SqZero}$ consisting of objects $(A,I,\delta)$ where
$A \in \Alg_{conn}$ and $I \in A\bimod^{\leq 0}$.

\subsection{Notation}\label{ss:rel-notation}

Throughout this section, for $\Psi:\Alg_{conn} \to \Sp$
and $B \to A$ a square-zero extension, we will use
the notation $\Psi(B/A)$ for $\Ker(\Psi(B) \to \Psi(A))$.

\subsection{}

We now introduce the following hypotheses on our functor 
$\Psi:\Alg_{conn} \to \Sp$.

\begin{defin}

$\Psi$ is
\emph{convergent} if for any $A \in \Alg_{conn}$,
the natural morphism:

\[
\Psi(A) \to \underset{n}{\lim} \, \Psi(\tau^{\geq -n} A)
\]

\noindent is an isomorphism.

\end{defin}

\begin{defin}

$\Psi$ \emph{infinitesimally commutes with sifted colimits}
if the functor:

\[
\begin{gathered}
\Alg_{conn}^{\SqZero} \to \Sp \\
(B \to A) \mapsto \Psi(B/A)
\end{gathered}
\]

\noindent commutes with sifted colimits.

\end{defin}

\begin{prop}\label{p:inf}

Suppose that $\Psi$ is convergent and infinitesimally commutes with
sifted colimits.
Suppose moreover that $\Psi$ is constant on split square-zero
extensions, i.e., for every $A \in \Alg_{conn}$ and
$M \in A\bimod^{\leq 0}$, the
morphism $\Psi(\SqZero(A,M)/A) = 0$.

Then $\Psi$ is \emph{infinitesimally constant}, that is,
for every $f:B \to A \in \Alg_{conn}$ with $H^0(B) \to H^0(A)$
surjective with nilpotent kernel, the map
$\Psi(B) \to \Psi(A)$ is an isomorphism.

\end{prop}

\begin{proof}

First, suppose $B \to A$ is a square-zero extension.

Note that if $A = T(V)$ is a \emph{free} $\sE_1$-algebra
on $V = \bS^{\oplus J}$ for some set $J$, then the 
then this square-zero extension is necessarily
a split square-zero extension. Indeed,
$\Ker(m)$ is then canonically isomorphic
to $T(V) \otimes V \otimes T(V)$ as a $T(V)$-bimodule;
in particular, it is free on $V$. It follows that
for $I \in T(V)\bimod^{-1}$, any morphism
$\Ker(m) \to I[1]$ is necessarily nullhomotopic.
In particular, our hypothesis implies $\Psi(B) \isom \Psi(A)$
for $A$ of this form.

In general, $A = |F_{\dot}|$ can be written
as a geometric realization of free $\sE_1$-algebras 
$F_n$ as above.\footnote{This follows
from the formalism of \cite{htt} \S 5.5.8, especially Lemma 5.5.8.14.}

 Then $B = |F_{\dot} \times_A B|$,
and each $F_n \times_A B \to F_n$ is a square-zero
extension, and this is a simplicial object of
$\Alg_{conn}^{\SqZero}$. Therefore, if
$\Psi$ infinitesimally commutes with sifted colimits,
we have:

\[
\Psi(B/A) =  
|\Psi(F_{\dot} \times_A B/F_{\dot})| = 
|0| = 0.
\]

\noindent This shows the result for general square-zero extensions.

Next, fix $A$. Recall (c.f. \cite{higheralgebra}
Corollary 7.4.1.28) that each morphism 
$\tau^{\geq -n -1}A \to \tau^{\geq -n} A$ admits a structure
of square-zero extension, so by the above
is an isomorphism on $\Psi$. By convergence, we then have:

\[
\Psi(A) \isom \underset{n}{\lim} \, \Psi(\tau^{\geq -n} A) \isom
\Psi(H^0(A)).
\]

Then for $f:B \to A$ surjective with nilpotent kernel,
we have $\Psi(B) = \Psi(H^0(B))$ and $\Psi(A) = \Psi(H^0(A))$.
Clearly $H^0(B) \to H^0(A)$ is a composition of square-zero
extensions, so we obtain the result.

\end{proof}

\subsection{Application to Dundas-Goodwillie-McCarthy}

Proposition \ref{p:inf} and the split square-zero case of
Theorem \ref{t:dgm} reduce the general case of
Theorem \ref{t:dgm} to the following to results.

\begin{thm}\label{t:conv/inf}

The functors $K,\TC:\Alg_{conn} \to \Sp$ are
convergent and infinitesimally commute with sifted colimits.

\end{thm} 

We prove this result separately for $K$-theory and $\TC$
below.

\subsection{$K$-theory}

First, convergence of $K$-theory is simple: the 
description of $K$-theory as a group completion
makes it clear that 
$\tau^{\geq -n-1} K(A) = \tau^{\geq -n-1} K(\tau^{\geq -n}A)$.

We will show $K$-theory infinitesimally commutes
with sifted colimits essentially following \cite{dgm}.
The proof uses \emph{Volodin's construction}, which is
an avatar of Milnor's definition of $K_2$ that
we review below.\footnote{I do not feel like I 
understand this method so well. Perhaps there is a more conceptual
approach.}

\subsection{Volodin theory}

Fix $A \in \Alg_{conn}$ in what follows. 
Recall that $\Proj(A)$ denotes the
category of projective $A$-modules, i.e., summands of $A^{\oplus n}$.
Let $\Proj(A)^{\simeq} \in \Gpd$ denote the underlying groupoid of this
category. Direct sums make $\Proj(A)^{\simeq}$ into an $\sE_{\infty}$-space,
and we recall that $\Omega^{\infty} K(A)$ is its group completion.

Note that if once we invert $A \in \Proj(A)^{\simeq}$, we obtain
a group. This motivates considering the canonical map:

\[
\chi = \chi_A:\underset{n \geq 0}{\colim} \, \Proj(A)^{\simeq} \to \Omega^{\infty} K(A) 
\in \Gpd
\]

\noindent where the structure maps in this colimit are given
by adding $A \in \Proj(A)$. Note that the left hand side
does not necessarily admit an $\sE_{\infty}$-structure, and
this map is not at all an equivalence. However, it is
an isomorphism on $\pi_0$.

Volodin's construction will provide a convenient expression for 
$\on{fib}(\chi)$ (the fiber over $0 \in \Omega^{\infty} K(A)$).

\subsection{}

For $n \geq 0$, let $GL_n(A)$ denote the group\footnote{Meaning
group-like $\sE_1$-groupoid.} of automorphisms
of $A^{\oplus n}$. 

Consider $A^{\oplus n}$ as a filtered $A$-module
in the standard way $0 \to A \to A^{\oplus 2} \to \ldots \to A^{\oplus n}$.
Let $B_n(A)$ be the group of automorphisms of
$A^{\oplus n}$ as a filtered $A$-module and let
$U_n(A)$ denote the kernel of the ``symbol" map:\footnote{I.e.,
the $i$th factor takes the induced automorphism
of $\gr_i A^{\oplus n} = A$.}

\[
B_n(A) \to \prod_{i=1}^n GL_1(A).
\]

\subsection{}\label{ss:volodin-funct}

We will need some care about the functoriality of the above
constructions. 

Let $\fSet^{inj}$ 
denote the category of finite sets and injective maps between them.
Then for any $I \in \fSet^{inj}$, we have 
$GL_I(A) \coloneqq \Aut(A^{\oplus I})$; and for $I \into J$, we have
a morphism $GL_I(A) \to GL_J(A)$ induced by the direct sum decomposition
$A^{\oplus J} = A^{\oplus I} \oplus A^{\oplus J\setminus I}$.

Relatedly, we have a functor from $\fSet^{inj} \to \Gpd$
sending $I$ to $\Proj(A)^{\simeq}$ for every
$I$, and sending a morphism $I \into J$ to the map 
$\Proj(A)^{\simeq} \xar{-\oplus A^{J \setminus I}} \Proj(A)^{\simeq}$.
Note that the source of the map $\chi$ is obtained by taking
the colimit of this construction along the map
$\bZ^{\geq 0} \xar{n \mapsto \{1,\ldots,n\}} \fSet^{inj}$.
(And one can show that the target is the colimit along
$\fSet^{inj}$, although we will not need this fact.)

The construction of $U_n$ depends also on a linear ordering,
i.e., for $I \in \bDelta_{aug}^{inj}$ a finite set with a linear ordering,
we have the group $U_I(A)$ functorial for injective
order preserving maps, and mapping naturally to
mapping to $GL_I(A)$. 

Moreover, the natural map: 

\[
\bB U_I(A) \to \bB GL_I(A) \to \Proj(A)^{\simeq} \to 
\Omega^{\infty} K(A)
\]

\noindent is canonically constant with constant value $A^{\oplus I}$.
Indeed, this is clear from the Waldhausen realization of $K(A)$.

\subsection{}

Putting the functoriality above
together, let $\sJ$ denote the
category (even poset) of
pairs $I \in \bDelta_{aug}^{inj}$ and
an isomorphism $\alpha:I \simeq \{1,\ldots,n\}$,
where morphisms $(I,\alpha) \to (J,\beta)$
are order-preserving maps $f:I \to J$ making
the diagram:

\[
\xymatrix{
I \ar[d]^{\alpha} \ar[r]^f & J \ar[d]^{\beta} \\
\{1,\ldots,n\} \ar@{^(->}[r] & \{1,\ldots,m\}
}
\]

\noindent commute, where the bottom arrow is the
standard embedding.

\begin{defin}

The \emph{Volodin space} 
$\sX(A)$ is the groupoid
$\colim_{(I,\alpha) \in \sJ} \bB U_I(A)$.

\end{defin}

We claim the observations from \S \ref{ss:volodin-funct} 
in effect equip $\sX(A)$ with a canonical map to
$\on{fib}(\chi)$. 

Indeed, note that
$\sJ$ maps to $\bZ^{\geq 0}$ (using the
isomorphism $\alpha$ from the pair $(I,\alpha)$),
inducing a map to
$\bB U_I(A) \to \bB GL_{|I|}(A) \to \Proj(A)^{\simeq}$.
On colimits, we obtain a map
to $\colim_{n \geq 0} \Proj(A)$, i.e., the source
of the map $\chi$. Then the 
Waldhausen realization of $K$-theory
shows that the composite map to
$\Omega^{\infty} K(A)$ is constant with
value the base-point.

\begin{thm}\label{t:volodin}

The map $\sX(A) \to \on{fib}(\chi)$ is an equivalence.

\end{thm}

\begin{rem}

The reader willing to take this result on faith may
safely skip ahead to \S \ref{ss:post-volodin}. 

\end{rem}

\begin{proof}[Proof of Theorem \ref{t:volodin} (sketch)]

This result is shown in \cite{fov} Proposition 2.6.
We give some indications of the ideas that go into it here.

Roughly, the construction works as follows. 
Note that $\sX(A)$ is connected (since $\sJ$ is contractible, admitting
an initial object). Let $\on{St}^{der}(A) \coloneqq \Omega \sX(A)$
be the \emph{derived Steinberg group} of $A$, so 
$\sX(A) = \colim_{(I,\alpha) \in \sJ} U_I(A)$, the colimit
being taken in the category $\mathsf{Gp}$ of group-like
$\sE_1$-groupoids. (It is straightforward to see from this description
that $\pi_0(\on{St}^{der}(A))$ is the usual Steinberg group
of $\pi_0(A)$, which is the reason we use this terminology.)

In essence, the argument relating $\on{St}^{der}(A)$ and $K(A)$ 
imitates the classical argument relating the Steinberg group and $K_2$
of usual ring.

\step First, suppose $G \in \mathsf{Gp}$ is a group (in the 
homotopy-theoretic sense). 

By definition, a \emph{central extension} of $G$ by 
an $\sE_2$-group\footnote{Meaning a group-like $\sE_2$-groupoid, i.e.,
a double loop space.} $A$ is the data of a pointed map
$\bB G \to \bB^2 A$. We frequently let $E$ denote the fiber
of the map $G \to \bB A$, which clearly fits into a fiber sequence
of groups:

\[
A \to E \to G.
\]

\noindent We sometimes say $A \to E \to G$ is a central extension
to mean it arises by this procedure.

\step We say $G$ is \emph{perfect} if
$\pi_0(G)$ is perfect in the usual sense, i.e.,
its (non-derived) abelianization is trivial.
We claim that perfect $G$ admits a \emph{universal central extension},
i.e., an initial central extension.

Indeed, in this case Quillen's plus construction\footnote{See \cite{hoyois}
for a discussion in the higher categorical setting.}
implies that there is simply-connected $Y \in \Gpd$ with a map from $\bB G$ 
realizing $Y$ as the initial simply-connected
space with a map from $\bB G$ (indeed: $Y = (\bB G)^+$).
Clearly $Y$ is initial for
pointed maps as well. Then looping implies $G$ has a universal
central extension by $\Omega^2 Y$.

We denote the universal central extension by $G^{univ} \to G$
in this case.

\step Next, we recall that there is a simple recognition principle
for a central extension $A \to E \to G$ to be the universal one.

Namely, this is the case if and only if $E$ is an \emph{acyclic group},
meaning that $\Sigma^{\infty}(\bB E) = 0 \in \Sp$
(or equivalently, the group homology of $E$ is trivial).

Indeed, this is a slight 
rephrasing of a well-known feature of Quillen's plus construction
(see \cite{dgm} Theorem 3.1.1.7 for example).

\step Now suppose $G$ is perfect and we are given a central extension:

\[
A \to H \to G
\]

\noindent with $H$ perfect as well. Then we claim the natural
map $H^{univ} \to G^{univ}$ is an isomorphism.

Indeed, let $A_G^{univ}$ denote $\Ker(G^{univ} \to G)$.
By universality, we have a canonical (pointed) map 
$\bB^2 A_G^{univ} \to \bB^2 A$. Let $F$ denote its
fiber. Note that $F$ receives a canonical map from $\bB H$
with fiber $\bB G^{univ}$. 

We claim $F$ is simply-connected. Indeed, its $\pi_1$ is
a quotient of $\pi_0(A)$, so abelian, but also a quotient
of the perfect group $\pi_0(H)$, so trivial.

Therefore, $F$ defines a central extension:

\[
\Omega^2 F \to G^{univ} \to H
\]

\noindent which is the universal one by the recognition principle:
$G^{univ}$ is an acyclic group, being the universal central
extension for $G$.

\step We now weaken the recognition principle given above
(c.f. \cite{dgm} Lemma 3.1.1.13).

Suppose we are given $f:E \to G$ surjective on $\pi_0$
with $G$ perfect and $E$ an acyclic group,
but without assuming centrality of the extension.
Observe that acyclicity of $E$ implies $E$ is perfect with
$E^{univ} \isom E$. Therefore, we obtain a map
$E^{univ} \to G^{univ}$, and we can ask when this map is an isomorphism.

We claim that this is the case if we assume that 
for every $g \in E$, the induced (conjugation)
automorphism of $\Ker(f)$ is trivial on $\pi_*(\Ker(f))$.

Indeed, we have the fiber sequence:

\[
\bB \Ker(f) \to \bB E \to \bB G
\]

\noindent with $\pi_1(\bB E) = \pi_0(E)$ acting trivially on
$\pi_*(\bB \Ker(f)) = \pi_{*+1}(\Ker(f))$. A well-known obstruction theoretic argument
then shows that for every $n \geq 0$, the map
$\tau_{\leq n} E \to \tau_{\leq n} G$ factors as a composition of
central extensions:

\[
\tau_{\leq n} E = H_0 \to H_1 \to \ldots \to H_r \to H_{r+1} = \tau_{\leq n} G
\in \mathsf{Gp}.
\]

\noindent The maps $H_i \to H_{i+1}$ are surjective on $\pi_0$ (being
extensions), so each group $H_i$ is perfect. Therefore, the previous
step implies that each such map induces an isomorphism on
universal central extensions. We then obtain
$(\tau_{\leq n} E)^{univ} \isom (\tau_{\leq n} G)^{univ}$.

Finally, it is clear that 
$G^{univ} = \lim_n (\tau_{\leq n} G)^{univ}$ (and similarly for $E$),
giving the claim.

\step We now apply these methods to study $\on{St}^{der}(A)$, leaving
details to references.

First, an argument of Suslin \cite{suslin} shows that
$\on{St}^{der}(A)$ is an acyclic group. 
More precisely, he shows that for any $(I,\alpha) \in \sJ$ 
and $r \geq 0$,
there is a map $(I,\alpha) \to (J,\beta) \in \sJ$
such that $\bB U_I(A) \to \bB U_J(A)$ is zero 
on homology groups $H_{\leq r}$ with coefficients in some field. 
This immediately implies the vanishing of such homology groups
for $\sX(A) = \bB \on{St}^{der}(A)$, which implies the vanishing
for integral coefficients, which gives our desired acyclicity. 

(We remark that although Suslin formally treats classical algebras,
his argument readily adapts to the connective $\sE_1$-setting.)  

\step Next, we would like to realize $\on{St}^{der}(A)$ as 
the universal central extension of something.

Let $GL_{\infty}(A) \coloneqq \colim_{n } GL_n(A) = 
\Omega(\colim_n \Proj(A))$.
Although $GL_{\infty}(A)$ is not perfect, it is not so far:
by Whitehead's lemma, $\pi_0(GL_{\infty}(A))$ has perfect derived
group an (non-derived) abelianization $K_1(A) (= K_1(\pi_0(A)))$.
Therefore, the kernel of the composition:
 
\[
GL_{\infty}(A) \to \Omega^{\infty+1} K(A) 
\xar{\Omega\chi} \pi_0(\Omega^{\infty+1} K(A)) =
K_1(A)
\]

\noindent is perfect. Let us denote\footnote{More standard notation
would be $E(A)$, but we have been using $E$ above for extensions.} 
this kernel by $G$.

There is a natural map $f:\on{St}^{der}(A) \to G$ induced
by our map from $\sX(A)$ to the fiber of $\chi$. We claim that
this realizes $\on{St}^{der}(A)$ as the universal central extension
of $G$.

Indeed, by acyclicity of $\on{St}^{der}(A)$, it suffices to show
that any $g \in \on{St}^{der}(A)$ acts trivially on $\pi_*(\Ker(f))$. 
This is a variant of the classical argument
that the usual Steinberg group is a central extension
\cite{milnor} Theorem 5.1, which we outline below
(see also \cite{fov} \S 4 and \cite{dgm} Lemma 3.1.3.4).

Let $\sJ_{\leq n} \subset \sJ$ be the full subcategory of
pairs $(I,\alpha)$ with $|I|\leq n$.
Let $\on{St}_n^{der}(A)$ be 
$\colim_{\sJ_{\leq n}} \on{St}_n^{der} \in \mathsf{Gp}$.

Note that $\on{St}_n^{der}(A)$ has a canonical map to $GL_n(A)$,
so acts on the spectrum $A^{\oplus n}$. 
Moreover, the map $\on{St}_n^{der}(A) \to \on{St}_{n+1}(A)$
clearly factors through $\on{St}_n^{der}(A) \rhd \Omega^{\infty} A^{\oplus n}$.

Let $K_n^{\prime}$ 
denote the kernel of the map $\on{St}_n^{der}(A) \to GL_n(A)$,
and let $K_{n+\frac{1}{2}}^{\prime}$ denote 
$\Ker(\on{St}_n^{der}(A) \rhd \Omega^{\infty} A^{\oplus n} \to GL_{n+1}(A))$.
Clearly the natural map $K_n^{\prime} \to K_{n+1}^{\prime}$ factors
through $K_{n+\frac{1}{2}}^{\prime}$.

It is immediate to see that 
for $x \in \Omega^{\infty} A^{\oplus n}$, the natural conjugation action
on $K_{n+\frac{1}{2}}^{\prime}$ fixes the image of
$K_n^{\prime}$, i.e., if we compose 
$\vph:K_n^{\prime} \to K_{n+\frac{1}{2}}^{\prime}$ with this conjugation
map, we obtain a map canonically homotopic to $\vph$.
This implies that the image of $x$ in $\pi_0(\on{St}^{der}(A))$ acts
by the identity on the image of $\pi_*(K_n^{\prime})$ in  
$\colim_m \pi_*(K_m^{\prime}) = 
\pi_*(\Ker(\on{St}^{der}(A) \to GL_{\infty}(A)))$.

The same arguments apply if we replace e.g.
$K_n^{\prime}$ by $K_n \coloneqq \Ker(\on{St}_n(A) \to GL_n(A) \times_{GL_{\infty}(A)} G)$. Then $\colim_n K_n$ is the
kernel of the map $\on{St}(A) \to GL_{\infty}(A)$
(by filteredness of the colimit here). 

Then observe that $\pi_0(\on{St}^{der}(A))$ is the usual Steinberg group
of $\pi_0(A)$, and the above shows that in the standard notation,
$x_{ij}(\alpha)$ acts trivially on the image of $\pi_*(K_n)$ in
$\pi_*(\on{St}^{der}(A))$ for $i<j$ and $j>n$. 
A similar argument shows the claim for $x_{ji}(\alpha)$ for
such $i$ and $j$. Such
elements then generate the (usual) Steinberg group, completing the
argument.

\step 

Let $G$ be as above. By Quillen's plus construction of $K$-theory,
the plus construction of $\bB G$ is 
$\Omega^{\infty} \tau^{\leq -2} K(A)$. 

Recall that the classifying space of the universal central
extension of $\bB G$ is the fiber of the canonical map
to the plus construction. Therefore, we have a diagram:

\[
\xymatrix{
\sX(A) = \bB \on{St}^{der}(A)  \ar[r] & \bB G \ar[r] \ar[d] & 
\Omega^{\infty} \tau^{\leq -2} K(A) \ar[d] \\
 & \bB GL_{\infty}(A) \ar[r] \ar[d] & 
\Omega^{\infty} \tau^{\leq -1} K(A) \ar[d] \\
& \underset{n}{\colim} \Proj(A) \ar[r] & 
\Omega^{\infty} K(A) 
}
\]

\noindent where the squares are all Cartesian and the top row
is a fiber sequence, implying the theorem.
(In particular, the universal central extension $\on{St}^{der}(A)$ 
of $G$ is by $\Omega^{\infty+2} K(A)$, in analogy with the classical
definition of $K_2$.)

\end{proof}

\subsection{Consequences}\label{ss:post-volodin}

We now deduce the following.

\begin{cor}

The functor:

\[
\begin{gathered}
\Alg_{conn} \to \Gpd \\
A \mapsto \on{fib}(\chi_A)
\end{gathered}
\]

\noindent commutes with sifted colimits.

\end{cor}

\begin{proof}

By Theorem \ref{t:volodin}, it suffices to show this
for $A \mapsto \sX(A)$ instead. We are
reduced to showing $A \mapsto \bB U_n(A)$ commutes with
sifted colimits for every $n$,
and then to $A \mapsto U_n(A) \in \Gpd$. But
$U_n(A)$ is functorially isomorphic to $A^{\oplus {{n-1}\choose{2}}}$,
which clearly commutes with sifted colimits in $A$.

\end{proof}

\subsection{}

We now show $K$-theory infinitesimally commutes with sifted colimits.

For convenience, we use some notation from the proof of Theorem \ref{t:volodin}
and let $\on{St}^{der}(A)= \Omega \sX(A) \in \mathsf{Gp}$. 
By Theorem \ref{t:volodin},
for any $A \in \Alg_{conn}$ we have a canonical homomorphism 
$\on{St}^{der}(A) \to GL_{\infty}(A)$ with:

\[
GL_{\infty}(A)/\on{St}^{der}(A) = \Omega^{\infty+1} K(A)
\]

\noindent (where the quotient on the left is the usual geometric realization
of the bar construction).

Similarly, we have:

\begin{lem}

Let $B \to A$ is a square-zero extension.

As in \S \ref{ss:rel-notation}, we let
$\on{St}^{der}(B/A) \coloneqq \Ker(\on{St}^{der}(B) \to \on{St}^{der}(A))$,
and similarly for $K(B/A)$ and $GL_{\infty}(B/A)$. 

Then:

\[
GL_{\infty}(B/A)/\on{St}^{der}(B/A) = \Omega^{\infty+1} K(B/A).
\]

\end{lem}

\begin{proof}

We have the commutative diagram:

\[
\xymatrix{
GL_{\infty}(B/A)/\on{St}^{der}(B/A) \ar[r] \ar[d] &
\bB \on{St}^{der}(B/A) \ar[r] \ar[d] &
\bB GL_{\infty}(B/A) \ar[d] \\ 
\Omega^{\infty+1} K(B) \ar[r] \ar[d] &
\bB \on{St}^{der}(B) \ar[r] \ar[d] &
\bB GL_{\infty}(B) \ar[d] \\ 
\Omega^{\infty+1} K(A) \ar[r] &
\bB \on{St}^{der}(A) \ar[r]  &
\bB GL_{\infty}(A)
}
\]

\noindent with rows being fiber sequences. Moreover, because
$\pi_0(\on{St}^{der}(B)) \onto \pi_0(\on{St}^{der}(A))$
and $\pi_0(GL_{\infty}(B)) \onto \pi_0(GL_{\infty}(A))$
(by the square-zero condition), the right two columns are also fiber sequences.
This gives the claim.

\end{proof}

Now observe that the functor:

\[
\begin{gathered}
\Alg_{conn}^{\SqZero} \to \Gpd \\
(B \to A) \mapsto GL_{\infty}(B/A) 
\end{gathered}
\]

\noindent commutes with sifted colimits. 
Indeed, the right hand 
side is isomorphic to infinite matrices with coefficients
in $\Omega^{\infty} \Ker(B \to A)$, so is a direct sum of infinitely many copies
of this functor.
Therefore, the claim follows from noting that 
$(B \to A) \mapsto \Omega^{\infty} \Ker(B \to A)$ commutes with sifted
colimits.

Mapping to $\on{St}^{der}(B/A)$ obviously commutes with sifted
colimits, so we obtain $(B\to A) \mapsto \Omega^{\infty+1}K(B/A)$
infinitesimally commutes with sifted colimits. Then the claim follows
because $\pi_0(\Omega^{\infty}K(B/A)) = 0$ and 
$\Omega^{\infty+1}: \Sp^{\leq -1} \to \Gpd$ commutes with sifted colimits.

\subsection{$\TC$}

We now show that $\TC$ is convergent and infinitesimally commutes
with sifted colimits.

\begin{notation}

In what follows, for a spectrum $V$ we let $V_{\bQ}$ denote $V \otimes \bQ$.

\end{notation}

\subsection{}

The following result, which
also serves as a key consistency check with 
Goodwillie's original work on the subject \cite{goodwillie},
is crucial.

\begin{thm}\label{t:rat'l}

For $B \to A \in \Alg_{conn}$ a square-zero extension, the natural map:

\[
\TC(B/A)_{\bQ} \to \TC(B_{\bQ}/A_{\bQ})
\]

\noindent is an isomorphism.

\end{thm}

We will prove this result in what follows.

\subsection{Notation}

Let $V$ be a cyclotomic spectrum. 

Let $\TC^-(V)$ denote $V^{h\bB \bZ}$.
Similarly, let $\TP(V) = V^{t\bB \bZ}$ be the
\emph{periodic topological cyclic homology} of $V$,
where this notation indicates the Tate
construction on the circle\footnote{We remind the reader of our
notational convention from Remark \ref{r:bb-bz}.}
(see e.g. \cite{nikolaus-scholze} \S I.4). 

Let $\TP^{\wedge}(V)$ denote the profinite completion
of $\TP(V)$, that is:

\[
\TP^{\wedge}(V) \coloneqq \Coker\big(\ul{\Hom}_{\Sp}(\bQ,\TP(V)) \to \TP(V)\big).
\]

Recall that we have:

\begin{equation}\label{eq:tc-tp}
\TC(V) = \on{Eq}(\TC^-(V) \underset{\vph}{\overset{\on{can}}{\rightrightarrows}} \TP^{\wedge}(V))
\end{equation}

\noindent by \cite{nikolaus-scholze} II.4.3.
Here $\on{can}$ is the tautological map (lifting
to $\TP$ itself), and $\vph$ is the \emph{cyclotomic Frobenius} map
(which amalgamates Frobenius maps at all primes).

\subsection{Construction of the meromorphic Frobenius}

We have the following observation.

\begin{lem}\label{l:tp-split}

Let $V$ be a spectrum with a (naive) $\bB\bZ$-action.
Suppose $(V_{\bQ})^{t\bB\bZ} = 0$. Then there is a canonical
isomorphism:

\[
(V^{h\bB \bZ})_{\bQ} \isom (V^{t\bB\bZ})_{\bQ} \times (V_{\bQ})^{h\bB \bZ}.
\]

\end{lem}

\begin{proof}

By definition of the Tate construction for $\bB \bZ$, we have
a commutative diagram with exact rows:

\[
\xymatrix{
(V_{h\bB \bZ})_{\bQ}[1] \ar[r] \ar[d] & (V^{h\bB \bZ})_{\bQ} \ar[r] \ar[d] &
(V^{t\bB \bZ})_{\bQ} \ar[d] \\
(V_{\bQ})_{h\bB \bZ}[1] \ar[r] & (V_{\bQ})^{h\bB\bZ} \ar[r] &
(V_{\bQ})^{t\bB \bZ}. 
}
\]

\noindent Clearly the left vertical map is an isomorphism,
so the right square is Cartesian. Now our vanishing hypothesis gives the result.

\end{proof}

We now make the following definition following \cite{hesselholt-zeta}.

\begin{defin}

A cyclotomic spectrum $V$ \emph{admits a meromorphic Frobenius}
if $\TP(V)$ is profinite\footnote{I.e., 
$\ul{\Hom}(\bQ,\TP(V)) = 0$, or equivalently, $\TP(V) \isom \TP^{\wedge}(V)$.}
and $\TP(V_{\bQ})$ is profinite.

\end{defin}

\begin{rem}

Note that $\TP(V_{\bQ})$ being profinite is equivalent to vanishing,
since a profinite and rational spectrum is necessarily zero.

\end{rem}

Suppose $V$ admits a meromorphic Frobenius.
We then have the canonical map:

\[
\vph^{mer}:\TP(V)_{\bQ} \to \TP(V)_{\bQ}
\]

\noindent given as the composition: 

\[
\TP(V)_{\bQ} 
\xar{\text{Lem. \ref{l:tp-split}}} \TC^-(V)_{\bQ}
\xar{\vph\otimes\bQ} \TP^{\wedge}(V)_{\bQ} = \TP(V)_{\bQ}.
\]

\begin{lem}\label{l:mer-inv}

Suppose $V$ is a cyclotomic spectrum that admits a meromorphic
Frobenius. Then $\TC(V)_{\bQ} \isom \TC(V_{\bQ})$
if and only if $\id-\vph^{mer}:\TP^{\wedge}(V)_{\bQ} \to 
\TP^{\wedge}(V)_{\bQ}$ is an isomorphism.

\end{lem}

\begin{proof}

We clearly have:

\[
\Ker(\id-\vph^{mer}) \isom 
\Ker\big(\TC(V)_{\bQ} \to \TC(V_{\bQ})\big)
\]

\noindent using Lemma \ref{l:tp-split}.

\end{proof}

\begin{rem}

$\THH(\bF_p)$ admits a meromorphic
Frobenius with non-invertible $\id-\vph^{mer}$:
see \cite{nikolaus-scholze} \S IV.4.

\end{rem}

\subsection{Filtrations}\label{ss:cyc-filt}

We now give some sufficient hypotheses to test the invertibility
of $\id-\vph^{mer}$ in the above setting.

\begin{defin}

A \emph{connectively filtered cyclotomic spectrum} 
is a connective spectrum $V$ with a complete\footnote{I.e.,
$\lim_n \fil_{-n} V = 0$.}
filtration $\fil_{\dot} V$
by connective spectra, an action of $\bB \bZ$ as a filtered spectrum, and 
$\bB \bZ$-equivariant filtered \emph{Frobenius}
maps:

\[
\vph_p:\fil_{\dot} \to \fil_{p\cdot\dot} V^{t\bZ/p}
\]

\noindent for every prime $p$. (As is standard, we are considering the 
residual\footnote{The quotient notation is misleading here:
it refers to the existence of the fiber sequence
$\bZ/p \to \bB \bZ \xar{p} \bB \bZ$ (with the right
map of course surjective on $\pi_0$).}
$\bB \bZ = (\bB \bZ)/(\bZ/p)$-action on the right hand side.)

\end{defin}

\begin{rem}

In the above, note that $V$ is a cyclotomic
spectrum. Moreover, the notation indicates that 
$\vph_p$ maps $\fil_n V$ to $\fil_{pn} V^{t\bZ/p}$
for every $p$ and $n$. Similarly, $\gr_{\dot} V = \oplus_n \gr_n V$ 
is also a \emph{graded} cyclotomic spectrum, which (similarly to
the proof of Proposition \ref{p:thh}) means 
$\vph_p:\gr_n V \to \gr_{pn} V^{t\bZ/p}$.

\end{rem}

\begin{rem}

The non-connective version of the above notion may readily
be extracted from \cite{amr}.

\end{rem}

\begin{rem}

A (non-connective) version of this notion 
appeared in \cite{brun-filtered} using the
language of equivariant homotopy theory.

\end{rem}

\begin{example}\label{e:filt-cyc}

If $A$ is a filtered $\sE_1$-algebra with 
$\fil_n A$ connective for all $n$, 
then the standard filtration on
$\THH(A)$ upgrades to a structure of connectively filtered cyclotomic spectrum. 
Indeed, this is a ready adaptation of the Nikolaus-Scholze construction
(or follows from the functoriality of traces outlined
in \S \ref{ss:tate-diag}). We have $\gr_{\dot} \THH(A) = \THH(\gr_{\dot} A)$
as (graded) cyclotomic spectra.

\end{example}

\begin{prop}\label{p:fil-inv}

Let $V$ be a connectively filtered cyclotomic spectrum
such that:

\begin{itemize}

\item $\fil_{-1} V \isom \fil_0 V \isom \ldots \isom V$.

\item For every $n$,
$\fil_{-m} V \in \Sp^{\leq -n}$ for $m\gg 0$. 

\item The cyclotomic spectrum $\gr_{\dot} V$ admits a meromorphic Frobenius.

\end{itemize}

Then $\TP(V)$ admits a meromorphic Frobenius
with $\id-\vph^{mer}:\TP(V) \otimes \bQ \to \TP(V) \otimes \bQ$
an isomorphism.

\end{prop}

From Lemma \ref{l:mer-inv}, we immediately deduce:

\begin{cor}\label{c:fil-rat'l}

In the setting of Proposition \ref{p:fil-inv}, 
$\TC(V)\otimes \bQ \isom \TC(V_{\bQ})$.

\end{cor}

\begin{proof}[Proof of Proposition \ref{p:fil-inv}]


\step 

First, note that $\bB \bZ$-equivariance of the
filtration on $V$, $\TC^-(V)$ and $V_{h\bB \bZ}$
inherit filtrations as well. These filtrations are
complete: for $\TC^-(V)$ this is automatic, and
for $V_{h\bB \bZ}$ this follows because the filtration 
on $\tau^{\geq -n} V_{h\bB \bZ}$ is
bounded from below for any $n$.

The norm map $V_{h\bB \bZ}[1] \to \TC^-(V)$ is clearly filtered,
so $\TP(V)$ also has a complete filtration. 
Note that $\gr_n \TP(V) = (\gr_n V)^{t\bB \bZ}$, and similarly
for the other players. In particular, $\gr_i \TP(V) = 0$ for
$n \geq 0$. 

In particular, we see that
$\ul{\Hom}(\bQ,\TP(V))$ has a complete filtration with
associated graded $\ul{\Hom}(\bQ,\gr_{\dot} \TP(V)) = 0$.
The same logic applies for $V_{\bQ}$, noting that
the filtration is complete for the same reason as for
$V_{h\bB\bZ}$. Therefore, we obtain
$\ul{\Hom}(\bQ,\gr_{\dot} \TP(V_{\bQ})) = 0$ 
and deduce that $V$ admits a meromorphic Frobenius.

\step 

Note that Lemma \ref{l:tp-split} applies just as well in the
setting of (connectively) filtered cyclotomic spectra.
Therefore, the meromorphic Frobenius on $\TP(V)_{\bQ}$
maps $\fil_n \TP(V)_{\bQ}$ ($=(\fil_n \TP(V))_{\bQ}$) to 
$\fil_{pn} \TP(V)_{\bQ}$.
Since $\fil_{-1} \TP(V)_{\bQ} = \TP(V)_{\bQ}$, we would
be done if this filtration on $\TP(V)_{\bQ}$ were complete.
But this is not typically true, 
so some additional argument is needed.

\step 

Fix $n$, and assume $m$ is large enough that
$\tau^{\geq -n} \fil_{-m} V = 0$.

We claim 
$\vph^{mer}:\tau^{\geq -n} \TP(V)_{\bQ} \to \TP(V)_{\bQ}$
is actually \emph{integral} on $\fil_{-m} V$.
That is, there are a canonical maps
$f_n:\tau^{\geq -n} \fil_{-m} \TP(V) \to 
\tau^{\geq -n} \fil_{-m} \TP(V)$
fitting into commutative diagrams:

\[
\xymatrix{
\tau^{\geq -n} \fil_{-m} \TP(V) \ar[rr]^{f_n} \ar[d] && 
\tau^{\geq -n} \fil_{-m} \TP(V) \ar[d] \\
\tau^{\geq -n} \fil_{-m} \TP(V)_{\bQ} \ar[rr]^{\vph^{mer}} && 
\tau^{\geq -n} \fil_{-m} \TP(V)
}
\]

\noindent and compatible in the natural sense as we suitably vary $n$ and $m$.

Indeed, our assumptions give $\fil_{-m} V_{h\bB \bZ} \in \Sp^{\leq -n}$,
so $\fil_{-m} V^{h\bB \bZ} \to \fil_{-m} V^{t\bB \bZ}$ induces an isomorphism
on $\tau^{\geq -n}$. Now our map $f_n$ is induced by the cyclotomic
Frobenius, noting that $\TP(V) = \TP^{\wedge}(V)$ as filtered
spectra by assumption. It is immediate to see that $f_n$ fits
into a commutative diagram as above.

\step 

Let $\TP_p^{\wedge}(V)$ be the $p$-adic completion of $\TP(V)$.
Note that by construction, $f_n$ induces a map
$\tau^{\geq -n} \fil_{-m} \TP(V) \to 
\tau^{\geq -n} \fil_{-pm} \TP_p^{\wedge}(V)$.

In particular, $f_n$ itself maps through 
$\tau^{\geq -n} \fil_{-2m} \TP(V)$.
We deduce that $\id-f_n:\tau^{\geq -n} \fil_{-m} \TP(V) \to
\tau^{\geq -n} \fil_{-m} \TP(V)$ is invertible (since the
filtration on $\TP(V)$ is complete).

Because tensoring with $\bQ$ is $t$-exact,
the map $\id-\vph^{mer}:\tau^{\geq -n} \fil_{-m} \TP(V)_{\bQ} \to
\tau^{\geq -n} \fil_{-m} \TP(V)_{\bQ}$ is 
an isomorphism. Since $\id-\vph^{mer}$ was an isomorphism   
on all associated graded terms, we now deduce that
it is an isomorphism after applying $\tau^{\geq -n}$ for any $n$,
and therefore is itself an isomorphism.

\end{proof}

\subsection{}

To prove Theorem \ref{t:rat'l}, it suffices (by Corollary \ref{c:fil-rat'l})
to show the following.

\begin{lem}\label{l:sq-zero-tp}

For $B \to A \in \Alg_{conn}$ a square-zero extension,
$\THH(B/A)$ admits a filtration as in Proposition \ref{p:fil-inv}.

\end{lem}

\begin{proof}

Suppose $B$ is a square-zero extension of $A$ by $M \in A\bimod^{\leq 0}$.
Note that $B$ admits a two-step descending
filtration with $\gr_{\dot} B = \SqZero(A,M)$ (more precisely:
$\gr_0 = A$ and $\gr_{-1} = M$).

As in Example \ref{e:filt-cyc}, $\THH(B/A)$ is a 
connectively filtered cyclotomic spectrum with 
$\gr_{\dot} \THH(B/A) = \THH(\SqZero(A,M)/A)$.
By connectivity of $A$ and $B$, we have
$\fil_{-n-1} \THH(B/A) \in \Sp^{\leq -n}$.
Moreover, it is clear that $\gr_i \THH(B/A) = 0$ for $i \geq 0$.

It remains to show that $\THH(\SqZero(A,M)/A)$ admits a meromorphic
Frobenius. As in the proof of Theorem \ref{t:tc-calc}, we have:

\[
\TP(\SqZero(A,M)/A) = \prod_{n \geq 1} \THH(A,M[1]^{\otimes n})^{t \bZ/n}.
\]

\noindent The $\bZ/n$-Tate construction
applied to connective spectra produces $n$-adically complete spectra,
showing that $\TP(\SqZero(A,M)/A)$ is profinite.
We then have: 

\[
\TP\big((\SqZero(A,M)/A)_{\bQ}\big) = 
\TP(\SqZero(A_{\bQ},M_{\bQ})/A_{\bQ})
\]

\noindent which is profinite by the above.

\end{proof}

\subsection{Sifted colimits}

We now show:

\begin{prop}

$\TC$ infinitesimally commutes with sifted colimits.

\end{prop}

\begin{proof}

It suffices to show that the functors $\TC(-)_{\bQ}$ and $\TC(-)/p$
infinitesimally commute with sifted colimits, where $p$ varies over all primes.
(Indeed, the functor $\Sp \to \Sp$ of tensoring with 
$\bQ \bigoplus\oplus_p \bS/p$ is continuous and conservative.)

In fact, the functor $\TC(-)/p:\Alg_{conn} \to \Sp$ (without relativization)
commutes with sifted colimits by 
\cite{cmm} Corollary 2.15.\footnote{In fact, it is clear from the proof that
in Theorem \ref{t:conv/inf}, we only needed $\Psi$ to infinitesimally
commute with geometric realizations, and here the argument is more
elementary: see \cite{cmm} Corollary 2.6.}

Now for $B \to A$ a square-zero extension, we have
$\TC(B/A)_{\bQ} = \TC(B_{\bQ}/A_{\bQ})$ by Theorem \ref{t:rat'l}.
Note that $\TC = \TC^-$ for rational ring spectra, so 
we have $\TC(B_{\bQ}/A_{\bQ}) = \TC^-(B_{\bQ}/A_{\bQ})$.
Finally, as $\TP(B_{\bQ}/A_{\bQ}) = 0$ by Lemma \ref{l:sq-zero-tp},
we have $\TC^-(B_{\bQ}/A_{\bQ}) = \THH(B_{\bQ}/A_{\bQ})_{h\bB \bZ}[1]$,
and this functor manifestly commutes with sifted colimits.

\end{proof}

\subsection{Convergence}

It remains to show that $\TC:\Alg_{conn} \to \Sp$ is convergent, 
which is straightforward. 

Note that $\THH:\Alg_{conn} \to \Sp$ is convergent, 
since $\tau^{\geq -n}\THH(A) \isom \tau^{\geq -n}\THH(\tau^{\geq -n} A)$
for every $n$. We formally obtain convergence of
$\TC^-$, since it is obtained as homotopy invariants from $\THH$.

Note that $\THH(-)_{h \bB \bZ}$ is convergent for the same reason
as $\THH$. Combined with the above, we obtain that $\TP$ is convergent.
This clearly implies that its profinite completion is convergent,
so we obtain the same property for $\TC$ from \eqref{eq:tc-tp}.

\bibliography{bibtex}{}

\begin{thebibliography}{AMGR}

\bibitem[AMGR]{amr}
David Ayala, Aaron Mazel-Gee, and Nick Rozenblyum.
\newblock {A naive approach to genuine $ G $-spectra and cyclotomic spectra}.
\newblock {\em arXiv preprint arXiv:1710.06416}, 2017.

\bibitem[Bar]{barwick-exact}
Clark Barwick.
\newblock {On exact $\infty$-categories and the theorem of the heart}.
\newblock {\em Compositio Mathematica}, 151(11):2160--2186, 2015.

\bibitem[Bei]{beilinson}
Alexander Beilinson.
\newblock {Relative continuous $K$-theory and cyclic homology}.
\newblock {\em arXiv preprint arXiv:1312.3299}, 2013.

\bibitem[BGT]{bgt}
Andrew~J Blumberg, David Gepner, and Gon{\c{c}}alo Tabuada.
\newblock {A universal characterization of higher algebraic $K$-theory}.
\newblock {\em Geometry \& Topology}, 17(2):733--838, 2013.

\bibitem[BHM]{bhm}
Marcel B{\"o}kstedt, Wu~Chung Hsiang, and Ib~Madsen.
\newblock {The cyclotomic trace and algebraic $K$-theory of spaces}.
\newblock {\em Inventiones mathematicae}, 111(1):465--539, 1993.

\bibitem[Blo]{bloch}
Spencer Bloch.
\newblock {On the tangent space to Quillen $K$-theory}.
\newblock In {\em Higher K-Theories}, pages 205--210. Springer, 1973.

\bibitem[BMS]{bms2}
Bhargav Bhatt, Matthew Morrow, and Peter Scholze.
\newblock {Topological Hochschild homology and integral $ p $-adic Hodge
  theory}.
\newblock {\em arXiv preprint arXiv:1802.03261}, 2018.

\bibitem[Bru]{brun-filtered}
Morten Brun.
\newblock {Filtered topological cyclic homology and relative $K$-theory of
  nilpotent ideals}.
\newblock {\em Algebraic \& Geometric Topology}, 1(1):201--230, 2001.

\bibitem[BS]{bhatt-scholze-grg}
Bhargav Bhatt and Peter Scholze.
\newblock {Projectivity of the Witt vector affine Grassmannian}.
\newblock {\em Inventiones mathematicae}, 209(2):329--423, 2017.

\bibitem[CMM]{cmm}
Dustin Clausen, Akhil Mathew, and Matthew Morrow.
\newblock {$K$-theory and topological cyclic homology of Henselian pairs}.
\newblock {\em arXiv preprint arXiv:1803.10897}, 2018.

\bibitem[DGM]{dgm}
Bj{\o}rn~Ian Dundas, Thomas Goodwillie, and Randy McCarthy.
\newblock {\em {The local structure of algebraic $K$-theory}}, volume~18.
\newblock Springer Science \& Business Media, 2012.

\bibitem[DM1]{dundas-mccarthy}
Bj{\o}rn~Ian Dundas and Randy McCarthy.
\newblock {Stable $K$-theory and topological Hochschild homology}.
\newblock {\em Annals of Mathematics}, 140(3):685--701, 1994.

\bibitem[DM2]{dundas-mccarthy-erratum}
Bj{\o}rn~Ian Dundas and Randy McCarthy.
\newblock {Erratum: Stable $K$-Theory and Topological Hochschild Homology}.
\newblock {\em Annals of Mathematics}, 142:425--426, 1995.

\bibitem[Fon]{fontes}
Ernest Fontes.
\newblock {A weighty theorem of the heart for the algebraic $K$-theory of
  higher categories}.
\newblock 2017.
\newblock Ph.D. thesis, available at
  \url{https://repositories.lib.utexas.edu/handle/2152/62096}.

\bibitem[FOV]{fov}
Zbigniew Fiedorowicz, C~Ogle, and Rainer~M Vogt.
\newblock {Volodin $K$-theory of $A_{\infty}$-ring spaces}.
\newblock {\em Topology}, 32(2):329--352, 1993.

\bibitem[Gai]{shvcat}
Dennis Gaitsgory.
\newblock Sheaves of categories and the notion of 1-affineness.
\newblock {\em Stacks and categories in geometry, topology, and algebra},
  643:127--225, 2015.

\bibitem[Goo1]{goodwillie}
Thomas~G Goodwillie.
\newblock {Relative algebraic $K$-theory and cyclic homology}.
\newblock {\em Annals of Mathematics}, 124(2):347--402, 1986.

\bibitem[Goo2]{goodwillie-icm}
Thomas~G. Goodwillie.
\newblock The differential calculus of homotopy functors.
\newblock In {\em Proceedings of the {I}nternational {C}ongress of
  {M}athematicians, {V}ol.\ {I}, {II} ({K}yoto, 1990)}, pages 621--630. Math.
  Soc. Japan, Tokyo, 1991.

\bibitem[GR]{grbook}
Dennis Gaitsgory and Nick Rozenblyum.
\newblock {\em {A study in derived algebraic geometry: Volume I:
  correspondences and duality}}.
\newblock Number 221. American Mathematical Soc., 2017.

\bibitem[Hes1]{hesselholt-stable-tc}
Lars Hesselholt.
\newblock {\em {Stable topological cyclic homology is topological Hochschild
  homology}}.
\newblock Matematisk Inst., Univ., 1993.

\bibitem[Hes2]{hesselholt-zeta}
Lars Hesselholt.
\newblock Topological {H}ochschild homology and the {H}asse-{W}eil zeta
  function.
\newblock In {\em An alpine bouquet of algebraic topology}, volume 708 of {\em
  Contemp. Math.}, pages 157--180. Amer. Math. Soc., Providence, RI, 2018.

\bibitem[Hoy]{hoyois}
Marc Hoyois.
\newblock {On Quillen's plus construction}, 2018.
\newblock Available at \url{http://www-bcf.usc.edu/~hoyois/papers/acyclic.pdf}.

\bibitem[KP]{kp}
Grigory Kondyrev and Artem Prihodko.
\newblock {Categorical proof of Holomorphic Atiyah-Bott formula}.
\newblock {\em arXiv preprint arXiv:1607.06345}, 2016.

\bibitem[LM]{lindenstrauss-mccarthy}
Ayelet Lindenstrauss and Randy McCarthy.
\newblock {On the Taylor tower of relative $K$-theory}.
\newblock {\em Geometry \& Topology}, 16(2):685--750, 2012.

\bibitem[Lur1]{htt}
Jacob Lurie.
\newblock {\em {Higher Topos Theory}}.
\newblock Princeton University Press, 2009.

\bibitem[Lur2]{higheralgebra}
Jacob Lurie.
\newblock Higher algebra.
\newblock Available at
  \url{http://math.harvard.edu/~lurie/papers/HigherAlgebra.pdf}, 2012.

\bibitem[McC]{mccarthy}
Randy McCarthy.
\newblock {Relative algebraic $K$-theory and topological cyclic homology}.
\newblock {\em Acta Mathematica}, 179(2):197--222, 1997.

\bibitem[Mil]{milnor}
John Milnor.
\newblock {\em {Introduction to Algebraic K-Theory}}, volume~72.
\newblock Princeton University Press, 2016.

\bibitem[Nik]{thomas-func}
Thomas Nikolaus.
\newblock {Topological Hochschild homology and cyclic $K$-theory}, 2018.
\newblock Preprint.

\bibitem[NS]{nikolaus-scholze}
Thomas Nikolaus and Peter Scholze.
\newblock {On topological cyclic homology}.
\newblock {\em arXiv preprint arXiv:1707.01799}, 2017.

\bibitem[SSW]{ssw}
Roland Schw{\"a}nzl, Ross Staffeldt, and Friedhelm Waldhausen.
\newblock {Stable $K$-theory and topological Hochschild homology of
  $A_{\infty}$ rings}.
\newblock In {\em Algebraic K-theory: Conference on Algebraic K-theory:
  September 4-8, 1995, the Adam Mickiewicz University, Pozn{\'a}n, Poland},
  volume 199, page 161. American Mathematical Soc., 1996.

\bibitem[Sus]{suslin}
AA~Suslin.
\newblock {On the equivalence of $K$-theories}.
\newblock {\em Communications in Algebra}, 9(15):1559--1566, 1981.

\bibitem[Wal]{waldhausen-stable}
Friedhelm Waldhausen.
\newblock Algebraic {$K$}-theory of topological spaces. {I}.
\newblock In {\em Algebraic and geometric topology ({P}roc. {S}ympos. {P}ure
  {M}ath., {S}tanford {U}niv., {S}tanford, {C}alif., 1976), {P}art 1}, Proc.
  Sympos. Pure Math., XXXII, pages 35--60. Amer. Math. Soc., Providence, R.I.,
  1978.

\end{thebibliography}
\bibliographystyle{alphanum}

\end{document}